\documentclass[reqno,oneside]{amsart}
\usepackage{amsfonts, amsmath, amsthm, amssymb}
\usepackage{float}
\usepackage[shortlabels]{enumitem}
\usepackage{geometry}
\usepackage[utf8]{inputenc}
\usepackage[square,numbers]{natbib}
\usepackage{nicefrac}
\usepackage{needspace}
\usepackage{tikz-cd}
\usepackage{xspace}
\usepackage{booktabs}
\usepackage[hidelinks]{hyperref}
\usepackage[capitalise,nosort]{cleveref}
\usepackage[colorinlistoftodos]{todonotes}

\newtheorem{theorem}{Theorem}[section]
\newtheorem{lemma}[theorem]{Lemma}
\newtheorem{proposition}[theorem]{Proposition}
\newtheorem{corollary}[theorem]{Corollary}

\theoremstyle{definition}
\newtheorem{definition}[theorem]{Definition}
\newtheorem{remark}[theorem]{Remark}

\newtheorem{claim}[theorem]{Claim}

\crefformat{footnote}{#2\footnotemark[#1]#3}

\newcommand{\ReDeclareMathOperator}[2]{\let#1\relax\DeclareMathOperator{#1}{#2}}

\DeclareMathOperator{\at}{\textit{at}}
\DeclareMathOperator{\pt}{\textit{pt}}

\ReDeclareMathOperator{\int}{int}
\ReDeclareMathOperator{\O}{O}
\ReDeclareMathOperator{\C}{C}
\ReDeclareMathOperator{\cmp}{cmp}
\ReDeclareMathOperator{\CL}{clp}
\ReDeclareMathOperator{\CLP}{Clp}
\ReDeclareMathOperator{\Uf}{Uf}
\ReDeclareMathOperator{\K}{K}
\ReDeclareMathOperator{\S}{Sat}
\ReDeclareMathOperator{\CS}{CSat}
\ReDeclareMathOperator{\KS}{KS}
\ReDeclareMathOperator{\LC}{LC}
\ReDeclareMathOperator{\LO}{CLC}
\ReDeclareMathOperator{\CSAT}{CSAT}
\ReDeclareMathOperator{\SFilt}{SFilt}
\ReDeclareMathOperator{\WLC}{WLC}
\ReDeclareMathOperator{\WLO}{WCLC}
\ReDeclareMathOperator{\GC}{AC}
\ReDeclareMathOperator{\P}{\mathcal{P}}

\newcommand{\categoryname}[1]{\ensuremath{\mathbf{#1}}\xspace}
\newcommand{\DeclareCategory}[2]{\newcommand{#1}{\categoryname{#2}}}
\DeclareCategory{\SMT}{SMT}
\DeclareCategory{\BA}{BA}
\DeclareCategory{\KMT}{KMT}
\DeclareCategory{\KHMT}{KHMT}
\DeclareCategory{\HMT}{HMT}
\DeclareCategory{\SobMT}{SobMT}
\DeclareCategory{\SobSMT}{SobSMT}
\DeclareCategory{\LCSobMT}{LCSobMT}
\DeclareCategory{\LCSobSMT}{LCSobSMT}
\DeclareCategory{\StLCSMT}{StLCSMT}
\DeclareCategory{\StLCMT}{StLCMT}
\DeclareCategory{\LCHMT}{LCHMT}
\DeclareCategory{\StKMT}{StKMT}
\DeclareCategory{\StKSMT}{StKSMT}
\DeclareCategory{\MT}{MT}
\DeclareCategory{\LStoneMT}{LStoneMT}
\DeclareCategory{\StoneMT}{StoneMT}
\DeclareCategory{\ZDMT}{ZMT}
\DeclareCategory{\StoneFrm}{StoneFrm}
\DeclareCategory{\StCFrm}{StCFrm}
\DeclareCategory{\StKFrm}{StKFrm}
\DeclareCategory{\Stone}{Stone}
\DeclareCategory{\LStone}{LStone}
\DeclareCategory{\LStoneFrm}{LStoneFrm}
\DeclareCategory{\ZDFrm}{ZFrm}
\DeclareCategory{\Frm}{Frm}
\DeclareCategory{\BFrm}{BFrm}
\DeclareCategory{\CFrm}{ContFrm}
\DeclareCategory{\SFrm}{SFrm}
\DeclareCategory{\RegFrm}{RegFrm}
\DeclareCategory{\CRegFrm}{ContRegFrm}
\DeclareCategory{\KRegFrm}{KRegFrm}
\DeclareCategory{\KHaus} {KHaus}
\DeclareCategory{\LCHaus} {LCHaus}
\DeclareCategory{\LCSob} {LCSob}
\DeclareCategory{\StLCSob} {StLCSob}
\DeclareCategory{\StKSob} {StKSob}
\DeclareCategory{\Top}{Top}
\DeclareCategory{\Sob}{Sob}
\DeclareCategory{\IA}{IA}
\DeclareCategory{\Reg}{Reg}
\DeclareCategory{\Set}{Set}
\DeclareCategory{\CABA}{CABA}
\DeclareCategory{\HA}{HA}
\DeclareCategory{\cHA}{cHA}

\newcommand{\nest}{\vartriangleleft}

\newcommand{\sq}{\underline{\mkern-1mu\square\mkern-1mu}\mkern2mu}

\renewcommand{\diamond}{\lozenge}
\let\ampersand\&
\renewcommand{\&}{\mathbin{\ampersand}}
\newcommand{\upset}{\mathord{\uparrow}}

\setlist[enumerate]{font=\normalfont}

\DeclareFontFamily{U}{mathb}{\hyphenchar\font45}
\DeclareFontShape{U}{mathb}{m}{n}{
<-6> mathb5 <6-7> mathb6 <7-8> mathb7
<8-9> mathb8 <9-10> mathb9
<10-12> mathb10 <12-> mathb12
}{}
\DeclareSymbolFont{mathb}{U}{mathb}{m}{n}
\DeclareMathSymbol{\llcurly}{\mathrel}{mathb}{"CE}
\DeclareMathSymbol{\ggcurly}{\mathrel}{mathb}{"CF}

\usetikzlibrary{arrows.meta}
\tikzset{
  symbol/.style={
    draw=none,
    every to/.append style={
      edge node={node [sloped, allow upside down, auto=false]{$#1$}}}
  }
}

\title{Local compactness in MT-algebras}
\author{Guram Bezhanishvili and Ranjitha Raviprakash}
\address{Department of Mathematical Sciences\\
New Mexico State University\\
Las Cruces NM 88003\\
USA}
\email{guram@nmsu.edu, prakash2@nmsu.edu}

\begin{document}

\begin{abstract}
In our previous work, we introduced McKinsey-Tarski algebras (MT-algebras for short) as an alternative pointfree approach to topology. Here we study local compactness in MT-algebras.  
We establish the Hofmann-Mislove theorem 
for sober MT-algebras, using which we develop the MT-algebra versions of such well-known dualities in pointfree topology as
Hofmann-Lawson, Isbell, and Stone dualities. This yields
a new perspective on these classic results. 
\end{abstract}

\maketitle

\tableofcontents 

\section{Introduction}

Pointfree topology is the study of topological spaces through their frames of open sets 
\cite{Johnstone, PicadoPultr2012, PicadoPultr2021}.
There is a classic dual adjunction between the categories $\Frm$ of frames and $\Top$ of topological spaces, which  
restricts to a dual equivalence between the full subcategories $\SFrm$ of spatial frames and 
$\Sob$ of sober spaces. Further restrictions yield the following well-known dualities in pointfree topology:
\begin{itemize}
	\item {\em Hofmann-Lawson duality} between the categories $\CFrm$ of continuous frames and $\LCSob$ of locally compact sober spaces \cite{HofmannLawson1978}.
	\item {\em Isbell duality} between the categories $\KRegFrm$ of compact regular frames and $\KHaus$ of compact Hausdorff spaces \cite{Isbell} (see also \cite{BanMul1980}).
	\item {\em Stone duality} between the categories $\BA$ of boolean algebras, $\StoneFrm$ of Stone frames, and $\Stone$ of Stone spaces \cite{Stone1936,banaschewski1989universal}.
\end{itemize}

Topological spaces can alternatively be studied using interior operators on their powerset algebras \cite{Kuratowski1922}. This approach was further developed by McKinsey and Tarski \cite{McKinseyTarski}, who proposed interior algebras as the {\em algebra of topology}. In \cite{BezhanishviliRR2023} we proposed complete interior algebras as an alternative pointfree approach to topology. We termed them McKinsey-Tarski algebras or simply MT-algebras. We thus arrived at the category 
$\MT$ of MT-algebras, and explored its connections to $\Frm$ and $\Top$, focusing on separation axioms within the context of MT-algebras. 

In this paper, we continue our study of MT-algebras, and mainly concentrate on local compactness in the setting of MT-algebras. We obtain an analogue of the Hofmann-Mislove theorem for sober MT-algebras, using which 
we establish the MT-versions of Hofmann-Lawson, Isbell, and Stone dualities. We also generalize the latter two to obtain dualities for locally compact Hausdorff and locally Stone MT-algebras, which yield dualities for locally compact regular and locally Stone frames. 

The paper is organized as follows. In \cref{sec: Prelims}, we recall the dual adjunction between $\Frm$ and $\Top$, 
and recap some basic facts about MT-algebras that will be used in the paper. In \cref{sec 3}, we develop the Hofmann-Mislove theorem for sober MT-algebras. In \cref{sec: H-L}, we define locally compact MT-algebras and obtain Hofmann-Lawson duality for locally compact sober spatial MT-algebras. In \cref{sec: 5}, we restrict the Hofmann-Lawson duality for MT-algebras to the setting of stably locally compact MT-algebras, which further restricts to the setting of stably compact MT-algebras. 
In \cref{sec: 6}, we show that these dualities restrict to locally compact Hausdorff MT-algebras. Further restriction yields Isbell duality for compact Hausdorff MT-algebras.
Finally, in \cref{sec: 7}, we define zero-dimensional, locally Stone, and Stone MT-algebras, and develop the MT-algebra version of Stone duality.

\section{Preliminaries} \label{sec: Prelims}

In this section, we recall the classic adjunction 
between frames and topological spaces. We also recall some basic facts about MT-algebras, 
which will play an important role in 
this paper.

\subsection{Frames}

A {\em frame} is a complete lattice $L$ satisfying the join-infinite distributive law
$$
a\wedge\bigvee S=\bigvee\{a\wedge s\mid s\in S\}
$$
for all $a\in L$ and $S\subseteq L$. A {\em frame homomorphism} is a map $h:L \to M$ between frames preserving finite meets and arbitrary joins. Let $\Frm$ be the category of frames and frame homomorphisms. 

We also let $\Top$ be the category of topological spaces and continuous maps. There is a well-known dual adjunction between $\Top$ and $\Frm$ (see, e.g., \cite[p.~16]{PicadoPultr2012}). 
The contravariant functor $\Omega :\Top \to \Frm$ sends a space $X$ to the frame $\Omega(X)$ of opens of $X$ and a continuous map $f:X\to Y$ to the frame homomorphism $\Omega(f):\Omega(Y)\to\Omega(X)$ given by $\Omega(f)(U)=f^{-1}[U]$ for each $U\in\Omega(Y)$. 

We recall that a {\em point} of a frame $L$ is a completely prime filter of $L$. The contravariant functor $\pt:\Frm \to \Top$ then sends each frame $L$ to the space $\pt(L)$ of points of $L$. The topology of $\pt(L)$ is given by $\zeta[L]$ where $\zeta(a)=\{ p\in\pt(L) \mid a\in p\}$ for each $a\in L$. Moreover, $\pt$ sends a frame homomorphism $h:L \to M$ to the continuous map $\pt(h):\pt(M)\to\pt(L)$ given by $\pt(h)(p)=h^{-1}[p]$ for each $p\in\pt(M)$. For $X\in\Top$ and $x\in X$, let $\delta(x)=\{ U\in\Omega(X) \mid x\in U \}$. 

\begin{theorem}[see, e.g., {\cite[pp.~18, 20]{PicadoPultr2012}}] \label{pointfree equivalence}
$ $
\begin{enumerate} [ref=\thetheorem(\arabic*)]

\item $\zeta : L \to \Omega(\pt(L))$ is an isomorphism iff $L$ is spatial.\label [theorem]{frames equivalence} 
\item $\delta : X \to \pt(\Omega(X))$ is a homeomorphism iff $X$ is sober.\label [theorem] {spaces equivalence}
\end{enumerate}
\end{theorem}

\subsection{MT-algebras} \label{sec: MT-algebras}

With each topological space $X$ we can also associate the powerset algebra $\P(X)$ equipped with the interior operator $\int$. Then $\Omega(X)$ is recovered as the fixed points of $\int$. This perspective originates in \cite{Kuratowski1922}. McKinsey and Tarski \cite{McKinseyTarski} generalized Kuratowski's approach to the study of interior operators on arbitrary boolean algebras. For further results in this direction we refer to \cite{Rasiowasikorski,Esakia,naturman1990interior}. The following definition and subsequent results are taken from \cite{BezhanishviliRR2023}. 

\begin{definition}[McKinsey-Tarski algebras]
   An \emph{MT-algebra} is a pair $M:=(B,\square)$ such that $B$ is a complete boolean algebra and $\square$ is an interior operator on $B$ (that is, $\square$ satisfies the Kuratowski axioms $\square 1=1$, $\square(a \wedge b) = \square a \wedge \square b$, $\square a \leq a$, and $\square a \leq \square \square a$ for all $a , b \in B$).
                   \end{definition}
 
 If $f: X \to Y$ is a continuous map,  
   then $f^{-1}: \mathcal{P}(Y) \to \mathcal{P}(X)$ is a complete boolean homomorphism satisfying $f^{-1}(\int(A)) \subseteq \int(f^{-1}(A))$.  This generalizes to the following: 
      
   \begin{definition}     A mapping $ f : M \to N$ between MT-algebras is an \emph{MT-morphism} if $f$ is a complete boolean homomorphism such that $f(\square_M a) \leq \square_N f(a)$ for all $ a \in M$.
    \end{definition}
    
 Let $\MT$ be the category of MT-algebras and MT-morphisms.

\begin{definition}\
Let $M$ be an MT-algebra. An element $a \in M$ is {\em open} if $a= \square a$. Let $\O(M)$ be the collection of open elements of $M$.
\end{definition}

\begin{lemma}
$ $
\begin{enumerate}[ref=\thelemma(\arabic*)]
\item If $M$ is an MT-algebra, then $\O(M)$ is a frame 
\item The restriction of an MT-morphism $f:M \to N$ to $f|_{\O(M)}:\O(M)\to\O(N)$ is a frame homomorphism. \label[lemma]{restriction to frame homomorphism}
\end{enumerate}
\end{lemma}

It follows that we have a functor $\O:\MT\to\Frm$. 
For $L\in\Frm$, let $(B(L),e)$ be the free boolean extension of $L$ (see, e.g., \cite[Sec.~V.4]{DL}). Then the right adjoint to the inclusion $e:L\to B(L)$ is an interior operator on $B(L)$ and $L$ is isomorphic to the fixpoints of $\square$. 

Let $\overline{B(L)}$ be the MacNeille completion of $B(L)$ (see, e.g., \cite[Sec.~XII.3]{DL}). Define the lower extension $\sq:\overline{B(L)}\to\overline{B(L)}$ of $\square:B(L)\to B(L)$ by 
\[
\sq x = \bigvee\{\square a\mid a\in B(L) \mbox{ and } a\leq x\}. 
\] 
Then $\left(\overline{B},\sq\right)$ is an MT-algebra, and we arrive at the following:

\begin{theorem} \label{thm: essentially surj}
The functor $\O:\MT\to\Frm$ is essentially surjective.
\end{theorem}

Clearly the assignment $X \mapsto (\P(X),\int)$ and $f \mapsto f^{-1}$ defines a contravariant functor $\P : \Top \to \MT$. In the other direction, for
$M \in \MT$, let $X = \at(M)$ be the set of atoms of $M$. Define 
$\eta : M \to \P(X)$ by 
$
\eta (a) =\{ x \in X \mid x \leq a \}.
$
Then $\eta$ is an onto complete boolean homomorphism. Therefore, the restriction of $\eta$ to $\O(M)$ is a frame homomorphism. Hence, the image $\tau:=\eta[\O(M)]$ is a topology on $X$. Moreover, for each MT-morphism $f:L\to M$, define 
$\at(f):\at(M)\to\at(L)$ by $\at(f)(x)=f^*(x)$, where $f^* : M \to L$ is the left adjoint of $f$ (given by $f^*(x)=\bigwedge\{a \in M \mid x \leq f(a)\}$). Then $\at(f)$ is a continuous map, and the assignment $M \mapsto \at(M)$ and $f \mapsto \at(f)$ defines a contravariant functor $\at:\MT \to \Top$.

\begin{theorem} \label{Top and MT}
$(\P, \at)$ is a dual adjunction between $\Top$ and $\MT$. 
\end{theorem}

\begin{remark}\label{MT adjunction}
The units of the dual adjunction $(\P,\at)$ are given by 
\[
\eta:1_{\MT}\to\P\circ\at \quad \mbox{and} \quad \varepsilon:1_{\Top}\to\at\circ\P,
\] 
where for  $M \in \MT$ and $X \in \Top$, we have 
\[
\eta(a) =\{ x \in \at(M) \mid x \leq a \} \quad \mbox{and} \quad \varepsilon (x) =\{ x \}.
\]
We then obtain that $\varepsilon$ is a homeomorphism for each $X\in\Top$, $\eta$ is an onto MT-morphism for each $M\in\MT$, and $\eta$ is an isomorphism iff $M$ is atomic.
\end{remark}

Since points of an MT-algebra correspond to its atoms, we call atomic MT-algebras {\em spatial}. Let $\SMT$ be the full subcategory of $\MT$ consisting of spatial MT-algebras. 
Then the dual adjunction 
of \cref{Top and MT} restricts to yield:

\begin{theorem}\label{MT equivalence}
$\Top$ is dually equivalent to $\SMT$.
\end{theorem}

\subsection{Lower separation axioms in MT-algebras}

We conclude this preliminary section with a brief account of lower separation axioms for MT-algebras, taken from \cite{BezhanishviliRR2023}. 

Let $M$ be an MT-algebra. An element $a\in M$  is {\em closed} if $a=\diamond a$, and \emph{locally closed} if it is a meet of an open element and a closed element. Let $\C(M)$ be the set of closed elements and $\LC(M)$ the set of locally closed elements of $M$.

We also call $a\in M$ \emph{saturated} if $a$ is a meet of open elements, and \emph{weakly locally closed} if 
it is a meet of a saturated element and a closed element. Let $\S(M)$ be the set of saturated elements and $\WLC(M)$ the set of weakly locally closed elements of $M$.

\begin{definition} Let $M\in\MT$. 
\begin{enumerate}
\item $M$ is a {\em $T_0$-algebra} if $\WLC(M)$ join-generates $M$.
\item $M$ is a {\em $T_{1/2}$-algebra} if $\LC(M)$ join-generates $M$. 
\item $M$ is {\em sober} if $M$ is a $T_0$-algebra and for each join-irreducible element $p$ of $\C(M) $ there is an atom $x$ in $M$ such that $p = \diamond x$.
 \end{enumerate}
\end{definition}

\begin{theorem}
Let $M\in\MT$.
\begin{enumerate} [ref=\thetheorem(\arabic*)]
\item $M$ is a $T_{1/2}$-algebra iff $M$ is isomorphic to $\overline{B(\O(M))}$.\label[theorem]{T_1/2} 
\item If $M$ is a sober $T_{1/2}$-algebra, then $M$ is spatial iff $\O(M)$ is spatial.\label [theorem]{ when is M spatial}
\item If $M$ is sober, then $\vartheta : \at(M) \to \pt(\O(M))$, defined by $\vartheta(x)= \upset x  \cap \O(M)$, is a homeomorphism.\label [theorem]{vartheta homeomorphism}
\end{enumerate}
\end{theorem}

Let  $\Sob$ be the full subcategory of $\Top$  consisting of sober spaces, and let $\SobMT$ be the full subcategory of $\MT$ consisting of sober MT-algebras. We then have: 

\begin{theorem} \label{thm: MTSob and Sob}
    The dual adjunction between $\Top$ and $\MT$ restricts to a dual adjunction between $\Sob$ and $\SobMT$, and the dual equivalence between $\Top$ and $\SMT$ restricts to a dual equivalence between $\Sob$ and $\SobSMT$.
\end{theorem}

\section{Sober MT-algebras and Hofmann-Mislove theorem}
\label{sec 3}
The Hofmann-Mislove theorem 
\cite{HofmannMislove1981} establishes an isomorphism between the poset of compact saturated subsets of a sober space $X$ and the poset of Scott open filters of the frame $\O(X)$ of open subsets of $X$. This isomorphism generalizes to an isomorphism between the poset of Scott open filters of an arbitrary frame $L$ and the poset of compact saturated subsets of the space $\pt(L)$ of points of $L$ 
\cite[Thm.~8.2.5]{Vickers}. 
In this section we further generalize it to the setting of MT-algebras.

We start by recalling the way-below relation (see, e.g., \cite[Def.~I.1.1]{Compendium}). 
Let $L$ be a frame and  
$a,b \in L$. We say that $a$ is {\em way-below} b and write $a \ll b$ provided for each $S\subseteq L$, from $b \leq \bigvee S$ it follows that $a \leq \bigvee T$ for some finite $T \subseteq S$. 
We next specialize the notion of an open filter of an interior algebra (see, e.g., \cite[p.~26]{Esakia}) to MT-algebras.

\begin{definition}
Let $M \in \MT$. We call a filter $F$ of $M$ an {\em open filter} if $a \in F$ implies $\square a \in F$. 
\end{definition}

It is well known (see, e.g., \cite[p.~26]{Esakia}) 
that the map $F \mapsto F \cap \O(M)$ is an isomorphism between the poset of open filters of $M$ and the poset of filters of $\O(M)$, both ordered by inclusion. Its inverse is give by the map $G \mapsto \upset G$. Thus,
$F$ is an open filter iff $F= \upset (F \cap \O(M))$.

\begin{definition}
    Let $M \in \MT$. We call an open filter $F$ of $M$ {\em Scott-open} if for each directed family 
            $S \subseteq \O(M)$, from $\bigvee S \in F$ it follows that $S \cap F \ne \varnothing$.
                      \end{definition}

\begin{remark} \label{rem: Scott-open filter}
Equivalently, an open filter $F$ of an MT-algebra $M$ is Scott-open provided for each $S\subseteq\O(M)$, from $\bigvee S \in F$ it follows that $\bigvee T \in F$ for some finite $T \subseteq S$.
\end{remark}

For $M\in\MT$, clearly $F$ is a Scott-open filter of $M$ iff $F\cap M$ is a Scott-open filter of $\O(M)$. Thus, the isomorphism between the posets of open filters of $M$ and filters of $\O(M)$ restricts to an isomorphism between the posets of Scott-open filters of $M$ and Scott-open filters of $\O(M)$. 

We require the following generalization of \cite [p.~302]{KiemelPaseka}:

\begin{theorem} [Keimel-Paseka Lemma] \label{KP} 
  Let $M \in \SobMT$. If $F$ is a Scott-open filter of $M$, then $\bigwedge F \leq u$ implies $u \in F$ for each $u \in \O(M)$. 
\end{theorem}

\begin{proof}
    Suppose there is $u \in \O(M)$ such that $\bigwedge F \leq u$ but $ u \notin F$. Let 
    \[
    Z=\{ v \in \O(M) \mid u \leq v \mbox{ and } v \notin F\}.
    \] 
    Clearly $Z\ne\varnothing$ since $u\in Z$. Also, $Z$ is inductive because $F$ is Scott-open. 
        Therefore, $\max Z\ne\varnothing$ by Zorn's Lemma. 
        Let $m\in\max Z$. We show that $m$ is meet-irreducible. Clearly $m\ne 1$ since $m\notin F$. Suppose that $m=a \wedge b$. Then 
        $a \wedge b \notin F$. Since $F$ is a filter, 
        $a\notin F$ or $b \notin F$. Thus, $a \in Z $ or $b \in Z$, 
        and hence $m=a$ or $m=b$. This proves that $m$ is meet-irreducible, which implies that $\neg m$ is join-irreducible in $\C(M)$. Because $M$ is sober, there is $x\in\at(M)$ such that $\neg m=\diamond x$, and so $m=\square \neg x$. 
    
    We show that $x\le a$ for each $a\in F$. If $x\not\le a$, then since $x$ is an atom, $x \leq \neg a$, so $a\le\neg x$, and hence $a\le \square\neg x=m$ because $a\in\O(M)$. Therefore, $a\notin F$, a contradiction. Thus, 
    \[
    x \le \bigwedge F \le u \le m = \square\neg x \le \neg x.
    \] 
    The obtained contradiction proves that such a $u$ does not exist.
 \end{proof}

The next definition is an obvious generalization of compactness to MT-algebras. 
  
 \begin{definition}
 Let $M$ be an MT-algebra.
 \begin{enumerate}
\item An element $a \in M$ is {\em compact} if for each $S \subseteq \O(M)$, from $a \leq \bigvee S$ it follows that $a \leq \bigvee T$ for some finite $T \subseteq S$. 
\item We call $M$ {\em compact} if its top element is compact.
\end{enumerate}
\end{definition}

The next result is a direct generalization to MT-algebras of the finite intersection property for compact spaces (see, e.g., \cite[Thm.~3.1.1]{Engelking}), so we skip the proof.

\begin{lemma}\label{FIP}
Let $M$ be an MT-algebra.
\begin{enumerate}[ref=\thelemma(\arabic*)]
\item  $a \in M$ is compact iff for each $F \subseteq \C(M)$ with $\left(\bigwedge F\right) \wedge a =0$, there is a finite $G \subseteq F$ with $\left(\bigwedge G\right) \wedge a=0$. \label[lemma]{compactelement property}
\item $M$ is compact iff  for each $F \subseteq \C(M)$ with $\bigwedge F =0$, there is a finite $G \subseteq F$ with $\bigwedge G=0$. \label[lemma]{compact property}
\end{enumerate}
\end{lemma}

For an MT-algebra $M$, let $\K(M)$ be the set of compact elements of $M$ and $\KS(M)$ the set of compact saturated elements of $M$. Let also $\SFilt(M)$ be the set of Scott-open filters of $M$.

 \begin{theorem} [Hofmann-Mislove for MT-algebras] \label{Hofmann-Mislove}
     Let $M\in \SobMT$. Then $(\SFilt(M),\subseteq)$ is isomorphic to $(\KS(M),\ge)$. 
 \end{theorem}
 
 \begin{proof}
    Define $\alpha : \KS(M) \to \SFilt(M)$ by $\alpha (s)= \{ a \in M \mid s \le \square a \}$. We first show that $\alpha$ is well defined. 
     
     \begin{claim}\label{porism up k}
    $\alpha(s) \in \SFilt(M)$. 
    \end{claim}
   
   \begin{proof} 
   It is straightforward to see that $\alpha(s)$ is an open filter. 
             To see that it is Scott-open, let $T\subseteq\O(M)$ be directed and $s\le\bigvee T$. Since $s$ is compact and $T$ is directed, $s\le t$ for some $t\in T$. Because $t\in\O(M)$, we have $t=\square t$. Thus, $t\in\alpha(s)$.
    \end{proof}       
       
    We next show that $\alpha$ preserves and reflects order.
    If $s,t\in\KS(M)$ with $s \ge t$, then it follows from definition that $\alpha(s) \subseteq \alpha(t)$. 
        For the converse, since $s$ is saturated, $s=\bigwedge\alpha(s)$, and similarly for $t$. Therefore, $\alpha(s) \subseteq \alpha(t)$ implies that $s=\bigwedge\alpha(s) \ge \bigwedge\alpha(t)=t$. Thus, $\alpha$ is an order-embedding. 
    It is left to prove that $\alpha$ is onto. 
    
        \begin{claim}\label{HM porism}
    If $F \in \SFilt(M)$ then $\bigwedge F \in \KS(M)$.
    \end{claim} 
    
    \begin{proof}
    We first show that $s:=\bigwedge F$ is compact. Let $T\subseteq\O(M)$ with $s\le\bigvee T$. By the Keimel-Paseka Lemma, $\bigvee T\in F$. Since 
        $F$ is Scott-open, $s\le \bigvee T_0$ for some finite $T_0\subseteq T$ (see \cref{rem: Scott-open filter}). Therefore, $s$ is compact. Moreover, since $\bigwedge F=\left(\bigwedge F\right) \cap \O(M)$, we conclude that $s$ is saturated. Thus, $s\in\KS(M)$. 
    \end{proof}
        
    Using the Keimel-Paseka Lemma yields that $\alpha(s) = F$. Consequently, $\alpha$ is onto, hence an order-isomoprhism.
    \end{proof}
        				
We next connect the Hofmann-Mislove theorem for frames 
to \cref{Hofmann-Mislove}. For this we require the following lemma. For a frame $L$, let $\SFilt(L)$ be the set of Scott-open filters of $L$.

\begin{lemma} \label{Hofmann-Mislove for frames}
 For a frame $L$, the posets $(\SFilt(L),\subseteq)$ and $(\SFilt(\Omega(\pt(L))),\subseteq)$ are isomorphic.
 \end{lemma}
 
  \begin{proof}
  Define $\phi:\SFilt(\Omega(\pt(L)) \to \SFilt(L)$ by $\phi(H)=\zeta^{-1}[H]$. It is straightforward to see that $\zeta^{-1}[H]$ is a filter of $L$. To see that $\zeta^{-1}[H]$ is Scott-open, let $S\subseteq L$ be a directed family such that $\bigvee S \in \zeta^{-1}[H]$. Then $\bigcup 
 \left(\{\zeta(s) \mid s \in S\}\right) = \zeta(\bigvee S) \in H$. Since $H$ is Scott-open, there is $t \in S$ such that $\zeta(t) \in H$. Hence, $t \in \zeta^{-1}[H]$, and so $\phi$ is well defined. 

 We next show that $\phi$ is an order-embedding. Let $H, G \in\SFilt(\Omega(\pt(L))$. Clearly, $F \subseteq G$ implies $\zeta^{-1}[H] \subseteq \zeta^{-1}[G]$. Conversely, since $\zeta$ is onto, $\zeta^{-1}(H) \subseteq \zeta^{-1}(G)$ implies 
 $
 H=\zeta[\zeta^{-1}(H)] \subseteq \zeta[\zeta^{-1}(G)] = G.
 $

Finally, we show that $\phi$ is onto. Let $F \in \SFilt(L)$. Since Scott-open filters are intersections of completely prime filters (see, e.g., \cite[Lem.~8.2.2]{Vickers}), $F = \bigcap \{x \in \pt(L) \mid F \subseteq x\}$. For each $x \in \pt(L)$, define $H_x = \{\zeta(a) \mid x \in \zeta(a)\}$. Then each $H_x$ is a completely prime filter of $\Omega(\pt(L))$. Let $H = \bigcap \{H_x \mid F \subseteq x\}$. Clearly $H$ is a filter of $\Omega(\pt(L))$. We show it is Scott-open. 
Let $S \subseteq L$ be such that $\{\zeta(s) \mid s\in S\}$ is directed and $\bigcup \{\zeta(s) \mid s\in S\} \in H$. Since $\zeta$ is a frame homomorphism, without loss of generality we may assume that $S$ is directed.
We have $\bigcup \{\zeta(s) \mid s\in S\} \in H_x$ for each $x\supseteq F$, so 
$x \in \bigcup\{\zeta(s) \mid s \in S\} = \zeta(\bigvee S)$, and hence $\bigvee S\in x$. Since this is true for each $x\supseteq F$, we obtain that $\bigvee S\in F$, so there is $s\in S\cap F$ because $F$ is Scott-open.
Therefore, $\zeta(s) \in \bigcap \{H_x \mid F \subseteq x\} = H$, and so $H$ is Scott-open. It remains to prove that $\zeta^{-1}[H] = F$. 
We have:
\begin{align*}
  a\in \zeta^{-1}[H] &\iff 
  \zeta(a) \in H \iff \zeta(a) \in H_x \ \ \forall x \supseteq F 
  \iff x \in \zeta(a) \ \ \forall x \supseteq F \\
  &\iff a \in x \ \ \forall x \supseteq F \iff a \in F.
\end{align*}
Thus, $\zeta^{-1}[H] = F$, as required.
\end{proof} 
  
  For a topological space $X$, let $\KS(X):=\KS(\P(X),\int)$ be the set of compact saturated subsets of $X$.

\begin{corollary}\ 
\begin{enumerate}[ref=\thecorollary(\arabic*)]
\item {\em (Hofmann-Mislove for frames)}
For a frame $L$, the posets $(\SFilt(L),\subseteq)$ and $(\KS(\pt(L)),\supseteq)$ are isomorphic. \label[corollary]{cor: HM for frames}
\item {\em (Hofmann-Mislove for sober spaces)}
For a sober space $X$, the posets $(\SFilt(\Omega(X)),\subseteq)$ and $(\KS(X),\supseteq)$ are isomorphic. \label[corollary]{cor: HM for spaces}
\end{enumerate}
\end{corollary}

\begin{proof}
(1) Since
$\pt(L)$ is a sober space, $\P(\pt(L))$ is a sober MT-algebra by \cref{thm: MTSob and Sob}. Therefore, by \cref{Hofmann-Mislove}, $(\SFilt(\P(\pt(L))),\subseteq)$ is isomorphic to $(\KS(\P(\pt(L))),\supseteq)=(\KS(\pt(L)),\supseteq)$. But 
$(\SFilt(\P(\pt(L))),\subseteq)$ is isomorphic to $(\SFilt(\O(\P(\pt(L))),\subseteq)=(\SFilt(\Omega(\pt(L)),\subseteq)$. Thus, the result follows
from \cref{Hofmann-Mislove for frames}

(2) If $X$ is a sober space, then $X$ is homeomorphic to $\pt(\Omega(X))$ (see, e.g., 
\cite[p.~20]{PicadoPultr2012}). Now apply~(1). 
\end{proof}

\begin{remark}
A natural approach to \cref{cor: HM for frames} would be to take 
$\overline{B(L)}$ and apply \cref{Hofmann-Mislove} to it. However, $\overline{B(L)}$ may not be a sober MT-algebra, and hence \cref{Hofmann-Mislove} may not apply. For example, consider $\mathbb{R}$ with the right ray topology $\tau = \{(a, \infty) \mid a \in \mathbb{R}\} \cup \{\varnothing,\mathbb{R}\}$. 
Since elements of $B(\Omega(X))$ are finite unions of sets of the form $(-\infty, a] \cap (b , \infty)$, where $a , b \in \mathbb{R}$, 
we see that $B(\Omega(\mathbb{R}))$ is atomless. Therefore, $\overline{B(\Omega(\mathbb{R})}$ is atomless. 
Thus, $\overline{B(\Omega(\mathbb{R})}$ is not sober because $\mathbb R$ is a join-irreducible closed element of $\overline{B(\Omega(\mathbb{R})}$ that is not the closure of any atom.
\end{remark}

 \section{Locally compact MT-algebras and Hofmann-Lawson duality} \label{sec: H-L}
 
  Hofmann-Lawson duality 
  \cite[Thm.~9.6]{HofmannLawson1978} establishes that the category of continuous frames 
  is dually equivalent to the category of locally compact sober spaces. 
   In this section we 
  develop a version of Hofmann-Lawson duality 
  for MT-algebras.

\begin{definition}     Let $M\in\MT$. 
    \begin{enumerate}
    \item For $a,b \in M$ define $a \nest b$ if there is $k \in \K(M)$ such that $ a \leq k \leq b$.

    \item We call $M$ \emph{locally compact} if for all $u \in \O(M)$ we have 
    $
    u=\bigvee \{v \in \O(M) \mid v \nest u \}.
    $
    \end{enumerate}
\end{definition}

\begin{remark}\label{nest on opens}
If $b$ in the above definition is open, then the $k\in K(M)$ such that $a\le k\le b$ can be chosen to be saturated. Indeed, set  
$k'=\bigwedge\{u \in \O(M) \mid k \leq u \}$. Then $k'$ is saturated  by definition, and it is compact because $k$ is compact. Moreover, $k'\le b$ since $b$ is open. Thus, $k'\in\KS(M)$ and 
$a \leq k' \leq b$.  
\end{remark}

The next lemma is straightforward to prove.

\begin{lemma} 
Let $M\in\MT$ and $a,b,c,d\in M$. \begin{enumerate}[ref=\thelemma(\arabic*)]
\item  $a,b \in \K(M) \implies a \vee b \in \K(M)$. \label[lemma]{join of compacts is compact}
\item $a \nest b \implies a \leq b$.
\item $a \leq b \nest c \leq d \implies a \nest d$. \label [lemma] {transitive}
\item $a \nest b$ and $c \nest d \implies (a \vee c) \nest (b \vee d)$. \label [lemma]{preserved under join}
\end{enumerate}
\end{lemma}

The next theorem is a direct generalization of \cite[Lem.~3.6]{HofmannLawson1978} to MT-algebras and we skip its proof.

\begin{theorem}\label{CF}
    Let $M\in\MT$ and $a,b\in\O(M)$.
    \begin{enumerate}[ref=\thelemma(\arabic*)]
    \item If $a \nest b$, then $a\ll b$. \label[theorem]{way below doesn't imply triangle}
    \item If $M$ is locally compact, then $a \nest b$ iff $a\ll b$. \label[theorem]{lcandcont}
        \end{enumerate}
\end{theorem}

\begin{remark}
In general, the converse of \cref{way below doesn't imply triangle} is false, already for spatial MT-algebras (see, e.g., \cite[p.~ 309]{Johnstone}).  
\end{remark}

Recall (see e.g., \cite[p.~135]{PicadoPultr2012}) that a frame $L$ is \emph{continuous} if $a = \bigvee \{x \in L \mid x \ll a\}$ for each $a \in L$.
      We call a frame homomorphism $f : L_1 \to L_2$ {\em proper}
      if $a \ll b$ implies $f(a) \ll f(b)$  for all $a,b \in L_1$. Let $\CFrm$ be the category of continuous frames and proper frame homomorphisms. The next result is crucial for proving Hofmann-Lawson duality, and will also play an important role in what follows. 
   
\begin{theorem} 
$ $
\begin{enumerate}[ref=\thetheorem(\arabic*)] 
\item {\em \cite[Lem.~VII.6.3.2]{PicadoPultr2012}} Let $L$ be a continuous frame and $b \ll a$ in $L$. Then there is a Scott-open filter $F$ of $L$ such that $a \in F$ and $F \subseteq \upset{b}$.\label[theorem]{lemma to prove spatial} 
    \item {\em \cite[Prop.~VII.6.3.3]{PicadoPultr2012}} Each continuous frame is spatial.\label[theorem]{continuous}
    \end{enumerate}    
\end{theorem}

      We relate locally compact MT-algebras to continuous frames.
      
\begin{theorem} \label{CFrm=LC}
    Let $M\in\MT$. 
    \begin{enumerate} [ref=\thelemma(\arabic*)] 
    \item If $M$ is locally compact, then $\O(M)$ is a continuous frame. \label[lemma]{LC implies CF} 
    \item If $M$ is sober and $\O(M)$ is a continuous frame, then $M$ is locally compact. \label[lemma]{CF implies LC}
    \end{enumerate}
\end{theorem}

\begin{proof}
    (1) Let $u \in \O(M)$. Since $M$ is locally compact, $u = \bigvee \{v \in \O(M) \mid v \nest u\}$, and hence $u = \bigvee \{v \in \O(M) \mid v \ll u\}$ by \cref{CF}(1). Thus, $\O(M)$ is a continuous frame.
    
    (2)     Let $u \in \O(M)$. Since $\O(M)$ is a continuous frame, $u=\bigvee\{v \in \O(M) \mid v \ll u \}$. We prove that $v \ll u$ implies $v \nest u$. 
        By \cref{lemma to prove spatial},
        from $v \ll u$ it follows that there is a Scott-open filter $F\subseteq \O(M)$ such that $u \in F$ and $F\subseteq{\uparrow}v$. 
                                    	By \cref{HM porism}, 
        $k:=\bigwedge F$ is a compact element of $M$. 
                        Therefore, $v \le k \le u$, which gives that $v \nest u$. Thus, $u = \bigvee \{v \in \O(M) \mid v \nest u\}$, and hence 
            $M$ is locally compact.
    \end{proof}

Let $X$ and $Y$ be locally compact sober spaces. Recall (see, e.g., \cite[Lem.~VI-6.21(ii)]{Compendium}) that a continuous map $f:X \to Y$ is {\em proper} if $ f^{-1}(K) $ is compact for each compact saturated $K \subseteq Y$. We generalize this to MT-morphisms. 

\begin{definition}\label{proper morphisms}
    Let $M,N$ be locally compact sober MT-algebras. We call an MT-morphism $f:M \to N$ {\em proper} if $a \in \KS(M)$ implies $f(a) \in \KS(N)$.
\end{definition}

\begin{remark} \label{proper}
Since every MT-morphism maps saturated elements to saturated elements, the condition in \cref{proper morphisms} is equivalent to $a \in \KS(M)$ implies $f(a) \in \K(N)$.
\end{remark}

Let $\LCSobMT$ be the category of locally compact sober MT-algebras and proper MT-morphisms, and let $\LCSobSMT = \LCSobMT \cap \SMT$. 
   
   \begin{lemma}\label{O}
       The restriction $\O : \LCSobMT\to \CFrm$ is well defined.
   \end{lemma}

   \begin{proof}
   That $\O$ is well defined on objects follows from 
       \cref{LC implies CF}. 
              We show that $\O$ is well defined on morphisms. Let $M,N \in \LCSobMT$ and $f: M \to N$ be a proper MT-morphism. By \cref{restriction to frame homomorphism},
              the restriction $f|_{\O(M)}:\O(M)\to\O(N)$ is a frame homomorphism. Suppose $a,b\in\O(M)$ with $a \ll b$. 
              Then $a \nest b $ by \cref{lcandcont}. Therefore, there is $k \in \K(M)$ with $a \leq k \leq b$. 
              Thus, $f(a) \leq f(k) \leq f(b)$. Since $f$ is proper, $f(k)\in \K(N)$, and hence $f(a) \nest f(b)$. Consequently, $f(a) \ll f(b)$, and so the restriction $f|_{\O(M)}$ is a proper frame homomorphism, yielding that $\O : \LCSobMT\to \CFrm$ is well defined.
       \end{proof}

       \begin{lemma} \label{lem: P circ pt}
       The composition $\P\circ \pt: \CFrm \to \LCSobMT$ is well defined.
       \end{lemma}
       
       \begin{proof}
       Let $L \in \CFrm$. 
              Then $\pt(L)$ is sober (see, e.g., \cite[p.~20]{PicadoPultr2012}),
              and hence $\P(\pt(L))$ is a sober MT-algebra by \cref{thm: MTSob and Sob}.
              Since $L$ is continuous, $L$ is spatial by \cref{continuous},  
                     and so $L \cong \Omega(\pt(L))$. Therefore, $\Omega(\pt(L))$ is continuous and \cref{CF implies LC} implies that $\P(\pt(L))$ is locally compact. Thus, $\P\circ\pt$ is well defined on objects. 
       
       To show that $\P\circ\pt$ is well defined on morphisms, let $L_1,L_2 \in \CFrm$ and $h: L_1\to L_2$ be a proper frame homomorphism. Then $f:=(h^{-1})^{-1} : \P(\pt(L_1)) \to \P(\pt(L_2))$ is an MT-morphism. To see that $h$ is proper, let $K \in \KS(\P(\pt(L_1)))$. Then 
       \begin{eqnarray*}
       f[K] &=& f\left[\bigcap\{\zeta(a) \mid K\subseteq\zeta(a)\}\right] = (h^{-1})^{-1}\left[\bigcap\{\zeta(a) \mid K\subseteq\zeta(a)\}\right] \\
       &=& \bigcap\left\{ (h^{-1})^{-1}\zeta(a) \mid K\subseteq\zeta(a)\right\} = \bigcap \left\{\zeta(h(a)) \mid K\subseteq\zeta(a) \right\},
                            \end{eqnarray*}
       where the last equality holds because $\zeta$ is 
       a natural transformation (see, e.g., \cite[p.~17]{PicadoPultr2012}). We show that $F:=\upset \{\zeta(h(a)) \mid K \subseteq \zeta(a) \} $ is a Scott-open filter of $\P(\pt(L_2))$. Clearly $F$ is an open filter. Suppose $\bigcup_i \zeta(m_i) \in F$. Then there is  $a \in L_1$ such that $K \subseteq \zeta(a)$ and $\zeta(h(a)) \subseteq \bigcup_i \zeta(m_i)$. Since $L_1$ is continuous, 
              $a = \bigvee \{ b \in L_1 \mid b \ll a \}$. Therefore, $K \subseteq \zeta\left(\bigvee \{ b \in L_1 \mid b \ll a \} \right) = \bigcup \{ \zeta(b) \mid b \ll a \}$, 
              so there exist $b_1,\dots,b_n \ll a$ such that $K \subseteq \bigcup_{i=1}^n \zeta(b_i) = \zeta\left(\bigvee_{i=1}^n b_i\right)$. Set $c = \bigvee_{i=1}^n b_i$. Then $K \subseteq \zeta(c)$ and $c \ll a$, so $h(c) \ll h(a)$ because $h$ is proper. But then $\zeta(h(c)) \ll \zeta(h(a))$ since $\zeta$ is an isomorphism. Thus, $\zeta(h(c)) \ll \zeta(h(a)) \subseteq \bigcup_i \zeta(m_i)$, and hence $\zeta(h(c)) \subseteq \bigcup_{j=1}^t \zeta(m_j)$. Consequently, $\bigcup_{j=1}^t \zeta(m_j)\in F$, and so $F$ is Scott-open. 
Since $f[K] = \bigcap F$, we conclude that $f[K]$ is compact saturated by \cref{HM porism}, and hence $f$ is a proper map. This yields that the composition $\P\circ \pt: \CFrm \to \LCSobMT$ is well defined.
       \end{proof}
        
       \begin{theorem}\label{CFrm adjunction} 
       \
       \begin{enumerate}[ref=\thetheorem(\arabic*)]
       \item The functors $\O$ and $\P\circ\pt$ are adjoint, $\O$ to the left and $\P\circ\pt$ to the
right, with the unit $(\P\vartheta)\circ\eta : 1_{\LCSobMT} \to(\P\circ\pt)\circ\O$ and co-unit $\zeta^{-1}:\O \circ(\P\circ\pt) \to 1_{\CFrm}$.
	\item The adjunction between $\LCSobMT$ and $\CFrm$ restricts to an equivalence between $\LCSobSMT$ and $\CFrm$. \label[theorem]{SMT_LS reflective subcategory}
	\end{enumerate}
\end{theorem}

\begin{proof}
(1) By \cref{O,lem: P circ pt}, the functors $\O:\LCSobMT\to\CFrm$ and $\P\circ\pt:\CFrm\to\LCSobMT$ are well defined. Let $M \in \LCSobMT$. Since $M$ is sober, $\vartheta(x)= \upset x \cap \O(M)$ is a homeomorphism by \cref{vartheta homeomorphism}.

\begin{claim}\label{natural transformation}
$\vartheta: \at \to \pt \circ \O$ is a natural isomorphism.
\end{claim}

\begin{proof}
Since each $\vartheta:\at(M)\to\pt(\O(M))$ is a homeomorphism, it is enough to show that for each proper MT-morphism $f:M \to N$, where $M, N \in \LCSobMT$, the diagram below commutes.
\begin{center}
      \begin{tikzcd}[row sep=4em]
      \at(N) \arrow[->]{rr}{\vartheta_{N}}
      \ar[d, "\at(f)"']& & \pt(\O(N))
      \ar[d, "\pt(\O(f))",]\\
       \at(M) \ar[rr, "\vartheta_{M}"']
        &&\pt(\O(M)) 
      \end{tikzcd}
\end{center}
For each $x \in \at(N)$ and $a\in\O(M)$, 
\begin{align*}
	a \in \pt(\O(f))(\vartheta_N)(x) &\iff a\in(f|_{\O(M)})^{-1}(\vartheta_N)(x)\iff a \in f^{-1}[\upset x \cap \O(N)] \\
	&\iff f(a) \in \upset x \cap \O(N) \iff x \le f(a) \iff f^*(x) \leq a \\
	&\iff a \in \upset{f^*(x)} \iff a \in \vartheta_{M}(f^*(x)).
\end{align*}	
	 Thus, $\pt(\O(f))\circ \vartheta_N =\vartheta_{M}\circ f^*$. 
\end{proof}

Since $\vartheta: \at \to \pt \circ \O$ is a natural isomorphism, $(\P\vartheta):\P \circ \at \to \P \circ (\pt \circ \O)$ is a natural isomorphism. By \cref{MT adjunction}, $\eta :1_{\LCSobMT} \to \P \circ \at$ is 
a natural transformation. Thus, $(\P\vartheta) \circ \eta : 1_{\LCSobMT} \to \P \circ (\pt \circ \O)$ is a natural transformation.

 On the other hand, it follows from \cite [p.~17]{PicadoPultr2012} that
 $\zeta: 1_{\CFrm} \to \Omega \circ \pt$ is a natural transformation. 
 Since $\zeta: L \to \Omega(\pt(L))$ is an isomorphism for each $L \in \CFrm$, we see that $\zeta^{-1}: \Omega \circ \pt \to 1_{\CFrm}$ is a natural isomorphism. But $\Omega \circ \pt = (\O \circ \P) \circ \pt$. Thus, $\zeta^{-1}: \O \circ (\P\circ\pt) \to 1_{\CFrm}$ is a natural isomorphism. 

(2) It suffices to show that, when restricted to $\LCSobSMT$ and $\CFrm$, the unit and co-unit 
are natural isomorphisms. As we already pointed out, 
the co-unit $\zeta^{-1}:\O\circ(\P\circ\pt) \to 1_{\CFrm}$ is a natural isomorphism. Moreover, 
$\P\vartheta$ is a natural isomorphism and, by \cref{MT adjunction}, $\eta :1_{\LCSobSMT} \to \P \circ \at$ is a natural isomorphism. Thus, the unit $(\P\vartheta)\circ\eta : 1_{\LCSobSMT} \to(\P\circ\pt)\circ\O$ is a natural isomorphism.
\end{proof}

As an immediate consequence of \cref{CFrm adjunction}, we obtain: 

\begin{theorem}
$\LCSobSMT$ is a reflective subcategory of $\LCSobMT$.
\end{theorem}

\begin{remark}
As we pointed out in \cref{continuous}, each continuous frame is spatial. While we expect that there exist non-spatial locally compact sober MT-algebras, such examples are still lacking.
\end{remark}

We conclude this section by establishing 
a dual adjunction between $\LCSobMT$ and the category $\LCSob$ of locally compact sober spaces and proper 
maps, which restricts to a dual equivalence between $\LCSobSMT$ and $\LCSob$. 
For this we require the following:

\begin{lemma}\label{eta injective on opens}
Let $M$ be a sober MT-algebra. If $\O(M)$ is a spatial frame, then $\eta: \O(M)\to \Omega(\at(M))$ is an isomorphism. 
\end{lemma}

\begin{proof}
 Since the restriction of $\eta$ to $\O(M)$ is a well-defined onto frame homomorphism (see \cref{sec: MT-algebras}), it is enough to show that $a\not\le b$ implies $\eta(a)\not\subseteq\eta(b)$ for each $a, b \in \O(M)$. Because $\O(M)$ is spatial, $a\not\le b$ implies that
there is $y \in \pt(\O(M)$ such that $a \in y$ and $b \notin y$. Since $M$ is sober, $\vartheta$ is a homeomorphism by \cref{vartheta homeomorphism}. Therefore, there is $x \in \at(M)$ such that $ y = \upset x \cap \O(M)$. Thus, $ x \leq a$ and $ x \nleq b$, and so $\eta(a) \not\subseteq \eta(b)$.
\end{proof}

\begin{lemma} \label{eta of Scott-open}
	Let $M \in \LCSobMT$. If $F$ is a Scott-open filter of $M$, then $\upset \eta[F]$ is a Scott-open filter of $\P(\at(M))$. 
\end{lemma}

\begin{proof}
	Clearly $\upset \eta[F]$ is an open filter. 
		Let $\{ a_i \} \subseteq \O(M)$ and $\bigcup_i \eta(a_i) \in \upset \eta[F]$. 
		Then there is $b \in F$ with $\eta(b) \subseteq \bigcup_i \eta(a_i)$. Since $F$ is an open filter, $\square b \in F$. We have 
	\[
	\eta(\square b) \subseteq \eta(b) \subseteq \bigcup_i \eta(a_i) = \eta\left(\bigvee_i a_i\right).
	\] 
Because $M$ is locally compact, $\O(M)$ is continuous by \cref{LC implies CF}. Therefore, 
  $\O(M)$ is spatial by \cref{continuous}. Since $M$ is sober, it follows from \cref{eta injective on opens} that the restriction of $\eta$ to $\O(M)$ is an isomorphism,
      hence $\square b \leq \bigvee_i a_i$. 
    Consequently, $\bigvee_i a_i \in F$. Since $F$ is Scott-open, 	there exist $i_1,\dots,i_n$ such that $\bigvee_{t=1}^n a_{i_t} \in F$. 
		But then $\bigcup_{t=1}^n \eta(a_{i_t}) = \eta\left(\bigvee_{t=1}^n a_{i_t}\right) \in \eta[F]$, and so
		$\upset \eta[F]$ is a Scott-open filter. 
\end{proof}

\begin{lemma}\label{eta onto on KS}
Let $M \in \MT$ and $K \in \KS(\at(M))$. Then $\bigwedge\{u \in \O(M) \mid K \subseteq \eta(u)\} \in \KS(M)$.
\end{lemma}

\begin{proof}
Let $U=\{u \in \O(M) \mid K \subseteq \eta(u)\}$. Clearly $\bigwedge U$ is saturated. To see that it is compact, let $\bigwedge U \leq \bigvee S$, where $S\subseteq\O(M)$. Then 
\[
K \subseteq \eta\left(\bigwedge U\right) \subseteq \eta\left(\bigvee S\right) = \bigcup \{ \eta(s) \mid s\in S \}.
\] 
Since $K$ is compact, 
$K \subseteq \bigcup \{ \eta(t) \mid T\subseteq S$ finite$\}$. Therefore, $K \subseteq \eta\left( \bigvee T\right)$, 
Thus, $\bigvee T\in U$, so   
$\bigwedge U  \leq \bigvee T$, and hence $\bigwedge U \in \K(M)$.
\end{proof}

Since $\eta:1_{\MT}\to\P\circ\at$ is a unit of the dual adjunction $(\P,\at)$ (see \cref{MT adjunction}), we also have: 

\begin{lemma}\label{eta and MT-morphism}
Let $M, N \in \MT$ and $f:M \to N$ be an MT-morphism. Then for each $a \in M$,
\[
\eta_N(f(a))=\at(f)^{-1}\eta_M(a).
\] 
\end{lemma}


\begin{lemma}\label{LS to LC}
The restriction $\at : \LCSobMT \to \LCSob$ is well defined.
\end{lemma}

\begin{proof}
Let $M \in \LCSobMT$. Then $\at(M)$ is sober by \cref{thm: MTSob and Sob}. We show that $\at(M)$ is locally compact. Let $a\in\O(M)$ and $x \in \eta(a)$. Then $x \leq a = \bigvee \{b \mid b \nest a\}= \bigvee \{b \mid b \ll a\}$, where the last equality follows from \cref{lcandcont}. Since $x$ is an atom, there is $b\ll a$ with $x \leq b$. By \cref{lemma to prove spatial}, there is a Scott-open filter $F\subseteq\O(M)$ such that $a\in F$ and $F\subseteq{\uparrow}b$. Let $G={\uparrow}F$. Then $G$ is a Scott-open filter of $M$.
By \cref{eta of Scott-open}, $\upset\eta[G]$ is a Scott-open filter of $\P(\at(M))$. 
By \cref{HM porism}, $\bigcap \upset \eta[G]$ is a compact saturated set of $\at(M)$. 
Moreover,
$x \in \eta(b) \subseteq \bigcap \upset \eta[G] \subseteq \eta(a)$. Thus, $\at(M)$ is locally compact, and hence the restriction of $\at$ is well defined on objects.

To see that it is also well defined on morphisms, suppose $M , N \in \LCSobMT $ and $f : M \to N$ is a proper MT-morphism. 
Since, $\at(f):\at(N) \to \at(M)$ is continuous, it is left to show that $\at(f)$ is proper. Let $K$ be a compact saturated subset of $\at(M)$. 
Since $K$ is saturated, $K=\bigcap\eta[U]$, where $U= \{u \in \O(M) \mid K \subseteq \eta(u)\}$. 
By \cref{eta onto on KS}, $\bigwedge U \in \KS(M)$. Since $f$ is proper, $f(\bigwedge U) \in \KS(N)$, 
so $\alpha( f(\bigwedge U))$ is a Scott-open filter of $N$ by \cref{porism up k}. Therefore, $\upset\eta[\alpha (f(\bigwedge U))]$ is a Scott-open filter by \cref{eta of Scott-open}. Hence, $\bigcap \upset \eta[\alpha (f(\bigwedge U))]$ is compact saturated by \cref{HM porism}. But 
\begin{align*}
\bigcap \upset \eta\left[\alpha\left( f\left(\bigwedge U\right)\right)\right] &= \bigcap \eta\left[\alpha \left(f\left(\bigwedge U\right)\right)\right] \\
&= \eta\left(\bigwedge\alpha\left( f\left(\bigwedge U\right)\right)\right)  \\
&= \eta\left(f\left(\bigwedge U\right)\right)  && \text{by \cref{Hofmann-Mislove} because } f\left(\bigwedge U\right) \in \KS(N) \\
&= \at(f)^{-1}\left(\eta\left(\bigwedge U\right)\right) && \text{by \cref{{eta and MT-morphism}}}\\ & = \at(f)^{-1}\left(K\right).
\end{align*}
Thus, $\at(f)^{-1}(K)$ is compact saturated, which shows that $\at(f)$ is proper.
\end{proof}

\begin{lemma}\label{LC to LS}
The restriction $\P:\LCSob \to \LCSobMT$ is well defined.
\end{lemma}

\begin{proof}
Let $X \in \LCSob$. Then $(\P(X),\int)$ is sober by \cref{thm: MTSob and Sob}. Moreover, since $X$ is locally compact, $U=\bigcup\{V \in \O(\P(X)) \mid V \nest U\}$ for every $U \in \O(\P(X))$. Therefore, $\P$ is well defined on objects. Let $X, Y \in \LCSob$ and $f : X \to Y$ be proper. Then $\P(f)=f^{-1}:\P(Y) \to \P(X)$ is an MT-morphism. 
Moreover, it is a proper MT-morphism since if $K \in \KS((\P(Y),\int))$, then 
$f^{-1}(K)\in\KS((\P(X),\int))$ because $f$ is a proper map.
Thus, 
$\P$ is also well defined on morphisms. 
\end{proof}

By putting \cref{thm: MTSob and Sob} together with \cref{LS to LC,LC to LS}, we obtain: 

\begin{theorem}[Hofmann-Lawson duality for MT-algebras]\label{LC equivalent to LS}
 The dual adjunction between $\Sob$ and $\SobMT$ restricts to a dual adjunction between $\LCSob$ and  $\LCSobMT$, and the dual equivalence between $\Sob$ and $\SobSMT$ restricts to a dual equivalence between $\LCSob$ and  $\LCSobSMT$.
\end{theorem}

\begin{proof}
Since $\varepsilon:X\to\at(\P(X))$ is a homeomorphism (see \cref{MT adjunction}), hence a proper map for all $X\in\LCSob$, it suffices to show that $\eta:M \to \P(\at(M))$ is a proper MT-morphism for all $M \in \LCSobMT$. 
Let $k \in \KS(M)$. Then $\alpha(k) \in \SFilt(M)$ by \cref{porism up k}, and $\upset{\eta[\alpha(k)]}$ is a Scott-open filter of $\P(\at(M))$ by \cref{eta of Scott-open}. By \cref{cor: HM for spaces},
$\bigcap \upset{\eta[\alpha(k)]}$ is compact. 
But
\[
	\bigcap \upset\eta[\alpha(k)] = \bigcap \eta[\alpha(k)] = \eta\left(\bigwedge \alpha(k)\right) = \eta(k),
\]
where the last equality holds because $\bigwedge \alpha(k) = k$ by 
\cref{Hofmann-Mislove}.
Thus, 
$\eta:M \to \P(\at(M))$ is a proper MT-morphism. 
\end{proof}

\begin{corollary}[Hofmann-Lawson duality] 
$\LCSob$ is dually equivalent to $\CFrm$.
\end{corollary}

\begin{proof}
Apply \cref{SMT_LS reflective subcategory,LC equivalent to LS}. 
\end{proof}

We thus arrive at the following diagram that commutes up to natural isomorphism: 

\begin{center}
      \begin{tikzcd}[row sep=2em]
      \LCSobMT \arrow[->]{rr}{\O}
      \ar[shift right=.5ex,"\at"',  shorten <= 5pt]{dr}
      & & \CFrm 
      \ar[dl, "pt", shift left=.5ex, shorten <= 5pt]
      \\
      &\LCSob 
      \ar[shift right=.5ex,"\P"', shorten >= 5pt]{ul} 
      \arrow[ur, "\Omega", shift left=.5ex,  shorten >= 5pt] &
      \end{tikzcd}
\end{center}
                                                                                                                   
\section{Stably locally compact and stably compact MT-algebras} \label{sec: 5}

Hofmann-Lawson duality restricts to a duality between the categories of stably continuous frames and stably locally compact spaces, which further restricts to a duality between the categories of stably compact frames and stably compact spaces (see \cite{GierzKeimel1977,Johnstone1980,Banaschewski1981,Simmons1982}). 
In this section we introduce stably locally compact and stably compact MT-algebras and show that \cref{CFrm adjunction,LC equivalent to LS} restrict to yield the corresponding results in the stably locally compact and stably compact settings. 

\begin{definition}\label{stably locally compact}
 Let $M$ be a locally compact sober MT-algebra. We call $M$ \emph{stably locally compact} provided $k,m \in \KS(M) \Longrightarrow k \wedge m \in \KS(M)$. 
\end{definition}

\begin{remark}
Since the meet of saturated elements is saturated, the above condition 
is equivalent to $k,m \in \KS(M) \Longrightarrow k \wedge m \in \K(M)$.
\end{remark}

Recall (see e.g., \cite[p.~488]{Compendium}) that a frame $L$ is {\em stably continuous} if $L$ is continuous and the way-below relation is stable (meaning that $a \ll b,c$ implies $ a \ll b \wedge c$ for all $a,b,c \in L$).

\begin{theorem} \label{StCFrm=StLC}
    Let $M\in\MT$. 
    \begin{enumerate} [ref=\thelemma(\arabic*)] 
    \item If $M$ is stably locally compact, then $\O(M)$ is a stably continuous frame.\label [lemma] {StCfrm implies StLC}
    \item If $M$ is sober and $\O(M)$ is a stably continuous frame, then $M$ is stably locally compact.\label [lemma] {StCF implies StLC}
    \end{enumerate}
\end{theorem}

\begin{proof}
(1) Since $M$ is locally compact, $\O(M)$ is a continuous frame by \cref{LC implies CF}. Let
 $a, b,c \in \O(M)$ with $a \ll b,c$. 
  By \cref{lcandcont}, $a \nest b,c$. Therefore, by \cref{nest on opens}, there exist $k,m \in \KS(M)$ such that $a \leq k \leq b$ and $a \leq m \leq c$. Because $M$ is stably locally compact, $k \wedge m \in \KS(M)$. Thus, $a \nest b \wedge c$, so $a \ll b \wedge c$, and hence $\O(M)$ is a stably locally compact frame.
 
 (2) Since $M$ is sober and $\O(M)$ is continuous, $M$ is locally compact by \cref{CF implies LC}. We show that $k,m \in \KS(M) \Longrightarrow k\wedge m\in\KS(M)$. By \cref{porism up k}, $k,m \in \KS(M)$ imply that $\alpha(k),\alpha(m) \in \SFilt(M)$. Let $F$ be the filter generated by $\alpha(k)$ and $\alpha(m)$. We show that $F \in \SFilt(M)$. Because $\alpha(k),\alpha(m)$ are open filters, so is $F$.
  Let $\bigvee S \in F$, where $S \subseteq \O(M)$. Then there exist $a,b\in M$ such that $k \leq \square a$, $m \leq \square b$, and $a\wedge b\le\bigvee S$. Since $\O(M)$ is continuous and $k\in \K(M)$, from $k \leq \square a$ it follows that $k \leq u \ll \square a$ for some $u \in\O(M)$. Similarly, $m \leq v \ll \square b$ for some $v \in \O(M)$.
 Because $\O(M)$ is stably continuous, 
  $u \wedge v \ll \square a \wedge \square b \leq \bigvee S$. Therefore, 
  there is a finite $T \subseteq S$ such that  $u \wedge v \leq \bigvee T$, so $\bigvee T \in F$ 
 since $u \wedge v \in F$. 
 Thus, $F\in\SFilt(M)$. 
  By \cref{HM porism}, $\bigwedge F \in \KS(M)$. 
  We show that $\bigwedge F = k \wedge m$. Clearly $k \wedge m \leq \bigwedge F$. For the reverse inequality, $\bigwedge F \leq \left(\bigwedge \alpha(k)\right) \wedge \left(\bigwedge \alpha(m)\right)$. Since 
  $k,m$ are saturated, $\bigwedge \alpha(k)=k$ and $\bigwedge \alpha (m)=m$. Thus, $\bigwedge F = k \wedge m$, so $ k \wedge m \in \KS(M)$ by \cref{HM porism}, and hence $M$ is stably locally compact.
 \end{proof} 

Let $\StLCMT$ be the full subcategory of $\LCSobMT$ consisting of stably locally compact MT-algebras.
Let also $\StCFrm$ be the full subcategory of $\CFrm$ consisting of stably continuous frames. 

\begin{lemma}\label{StMT to STFrm}
       The restriction $\O : \StLCMT\to \StCFrm$ is well defined.
   \end{lemma}
   
   \begin{proof}
   Apply \cref{O,StCfrm implies StLC}.
      \end{proof}
   
       \begin{lemma}\label{STFrm to StMT}
       The restriction of the composition $\P\circ \pt: \StCFrm \to \StLCMT$ is well defined.
       \end{lemma}
       
       \begin{proof}
       Apply \cref{lem: P circ pt,StCF implies StLC}.
       \end{proof}

\begin{theorem}\label{St adjunction}
The adjunction between $\LCSobMT$ and $\CFrm$ restricts to an adjunction between $\StLCSMT$ and $\StCFrm$.
The equivalence between $\LCSobSMT$ and $\CFrm$ restricts to an equivalence between $\StLCSMT$ and $\StCFrm$.
\end{theorem}

\begin{proof}
Apply \cref{CFrm adjunction,StMT to STFrm,STFrm to StMT}.
\end{proof}

\begin{lemma}\label{StLS to StLC}
The restriction $\at : \StLCMT \to \StLCSob$ is well defined.
\end{lemma}

\begin{proof}
In view of \cref{LS to LC}, 
we only need to show that $M\in \StLCMT$ implies $\at(M)$ is stably locally compact. 
For this it is sufficient to show that the intersection of two compact saturated sets of $\at(M)$ is compact. Let $K, H$ be compact saturated sets of $\at(M)$. Then $K=\eta[\bigwedge S]$ and $H=\eta[\bigwedge T]$, where $S,T\subseteq\O(M)$. Set $k=\bigwedge S$ and $m=\bigwedge T$. Then $k,m \in \KS(M)$ by \cref{eta onto on KS} and $K \cap H = \eta(k \wedge m)$. Since $M$ is stably locally compact, $k \wedge m \in \KS(M)$. By \cref{porism up k}, $\alpha(k\wedge m) \in \SFilt(M)$. By \cref{eta of Scott-open}, $\upset \eta[\alpha(k\wedge m)]$ 
 is a Scott-open filter of $\P(\at(M))$. Thus, $\bigcap \upset \eta[\alpha(k\wedge m)]$ is compact by \cref{HM porism}. But $\bigcap \upset \eta[\alpha(k\wedge m)]=\eta(k\wedge m)=K\cap H$ by \cref{Hofmann-Mislove}.
\end{proof}

\begin{lemma}\label{StLC to StLS}
The restriction $\P:\StLCSob \to \StLCMT$ is well defined.
\end{lemma}

\begin{proof}
This is immediate from the definition of stably locally compact spaces and \cref{LC to LS}.
\end{proof} 

\begin{theorem}\label{StLC equivalent to StLS}
The dual adjunction between $\LCSob$ and $\LCSobMT$ restricts to a dual adjunction between $\StLCSob$ and  $\StLCMT$, and the dual equivalence between $\LCSob$ and $\LCSobSMT$ restricts to a dual equivalence between $\StLCSob$ and  $\StLCSMT$.
\end{theorem}

\begin{proof}
Apply \cref{LC equivalent to LS,StLS to StLC,StLC to StLS}.
\end{proof}

As an immediate consequence, 
we arrive at the first result mentioned at the beginning of this section:

\begin{corollary}
$\StCFrm$ is dually equivalent to $\StLCSob$.
\end{corollary}

\begin{proof}
Apply \cref{St adjunction,StLC equivalent to StLS}.
\end{proof}

We thus arrive at the following diagram that commutes up to natural isomorphism:

\begin{center}
      \begin{tikzcd}[row sep=2em]
      \StLCMT \arrow[->]{rr}{\O}
      \ar[shift right=.5ex,"\at"',  shorten <= 7pt]{dr}
      & & \StCFrm 
      \ar[dl, "pt", shift left=.5ex, shorten <= 7pt]
      \\
      &\StLCSob 
      \ar[shift right=.5ex,"\P"', shorten >= 7pt]{ul} 
      \arrow[ur, "\Omega", shift left=.5ex,  shorten >= 7pt] &
      \end{tikzcd}
\end{center}
We conclude this section by specializing the above results to the stably compact setting.

\begin{definition}
Let $M$ be a stably locally compact MT-algebra. We call $M$ {\em stably compact} if $M$ is compact.
\end{definition}
 
 Let $\StKMT$ be the full subcategory of $\StLCMT$ consisting of stably compact MT-algebras and $\StKSMT = \StKMT \cap \SMT$. Let also $\StKFrm$ be the full subcategory of $\StCFrm$ consisting of stably continuous frames. It follows from the definition of compactness that
  an MT-algebra $M$ is compact iff $\O(M)$ is a compact frame. We thus arrive at the following: 
  
 \begin{theorem}\label{StKMT and StKFRm}
 The adjunction between $\StLCMT$ and $\StCFrm$ restricts to an adjunction between $\StKMT$ and $\StKFrm$, and
the equivalence between $\StLCSMT$ and $\StCFrm$ restricts to an equivalence between $\StKSMT$ and $\StKFrm$.
 \end{theorem}
 
 \begin{lemma}\label{compactness of atoms}
 Let $M \in \LCSobMT$. Then $M$ is a compact MT-algebra iff $\at(M)$ is a compact space.
 \end{lemma}
 
 \begin{proof}
  First suppose that $M$ is compact. Let $\at(M) = \bigcup \eta[S]$, where $S\subseteq\O(M)$. Then, $\eta(1) = \eta(\bigvee S)$. Since $M$ is locally compact, $\O(M)$ is a continuous frame by \cref{LC implies CF}. Therefore, $\O(M)$ is spatial by \cref{continuous}. Moreover, because $M$ is sober, \cref{eta injective on opens} implies that the restriction of $\eta$ to $\O(M)$ is an isomorphism, and hence $1= \bigvee S$. 
  Since $M$ is compact, $1=\bigvee T$ for some finite $T\subseteq S$. 
  Thus, $\at(M)=\eta(1)=\bigcup \eta[T]$, and so $\at(M)$ is compact. 
 
For the converse, suppose that $\at(M)$ is a compact space. Let $1= \bigvee S$, where $S \subseteq \O(M)$. Then $ \at(M)=\eta(1)=\eta(\bigvee S)=\bigcup \eta[S]$. Since $\at(M)$ is compact, $\at(M)= \bigcup \eta[T]=\eta(\bigvee T)$ for some finite $T \subseteq S$. Using \cref{eta injective on opens} again, we get $1 = \bigvee T$. Thus, $M$ is a compact MT-algebra.
 \end{proof}
 
By putting \cref{StLC equivalent to StLS,compactness of atoms} together, we obtain:
 
 \begin{theorem}\label{StKMT equivalent to StK}
 The dual adjunction between $\StLCSob$ and $\StLCMT$ restricts to a dual adjunction between $\StKSob$ and  $\StKMT$, and the dual equivalence between $\StLCSob$ and $\StLCSMT$ restricts to a dual equivalence
between $\StKSob$ and  $\StKSMT$.
\end{theorem}

As an consequence, we arrive at the second result mentioned at the beginning of this section:

\begin{corollary}
$\StKFrm$ is dually equivalent to $\StKSob$. 
\end{corollary}

\begin{proof}
Apply \cref{StKMT and StKFRm,StKMT equivalent to StK}.
\end{proof}

We thus arrive at the following diagram that commutes up to natural isomorphism:

\begin{center}
      \begin{tikzcd}[row sep=2em]
      \StKMT \arrow[->]{rr}{\O}
      \ar[shift right=.5ex,"\at"',  shorten <= 7pt]{dr}
      & & \StKFrm 
      \ar[dl, "pt", shift left=.5ex, shorten <= 7pt]
      \\
      &\StKSob 
      \ar[shift right=.5ex,"\P"', shorten >= 7pt]{ul} 
      \arrow[ur, "\Omega", shift left=.5ex,  shorten >= 7pt] &
      \end{tikzcd}
\end{center}

\section{Locally compact Hausdorff MT-algebras}\label{sec: 6}
In this section we restrict equivalences and dual equivalences of the previous two sections to the Hausdorff setting. Among other things, this yields an alternative approach to Isbell duality between the categories of compact Hausdorff spaces and compact regular frames.

We begin by recalling 
the definitions of $T_1$ and regular MT-algebras.
An MT-algebra $M$ is a {\em $T_1$-algebra} if $\C(M) $ join-generates $M$. Equivalently, $M$ is a {\em $T_1$-algebra} provided for all $a \neq 0$, there exists $c \in \C(M)$ such that $0 \neq c \leq a$. 

An MT-algebra $M$ is {\em regular} 
if $M$ is a $T_1$-algebra and for all $a \in \O(M)$ we have 
\begin{equation*}
a = \bigvee \{b \in \O(M) \mid \diamond b \leq a\}. 
\end{equation*}

To define Hausdorff MT-algebras, we recall 
that an element $a$ of an MT-algebra $M$ is {\em regular closed} if $a = \diamond \square a$, and that $a$ is \emph{approximated from above} by regular closed elements if 
\[
a=\bigwedge \{\diamond \square c \mid a \leq \square c\}.
\] 
Let $\GC(M)$ be the collection of such elements of $M$. Then an MT-algebra $M$ is {\em Hausdorff} 
if $\GC(M)$ join-generates~$M$. 
The following two lemmas are proved in 
\cite[Sec.~6, 7]{BezhanishviliRR2023}. 

\begin{lemma} \label{lem: separation properties}
Let $M$ be an MT-algebra.
\begin{enumerate} [ref=\thelemma(\arabic*)]
\item If $M$ is regular, then $M$ is Hausdorff. \label[lemma]{T3 implies T2} 
 \item If $M$ is Hausdorff, then $M$ is sober and $T_1$. \label[lemma]{T2 implies T1}
\item If $M$ is $T_1$, then $M$ is $T_{1/2}$. \label [lemma]{T1 implies T1/2}

\end{enumerate}
\end{lemma}

\begin{lemma}\label{O is ess surj on reg}
$ $
\begin{enumerate}[ref=\thelemma(\arabic*)]
\item If $M$ is a regular MT-algebra, then $\O(M)$ is a regular frame.\label[lemma]{regular algebra to frm}
\item If $L$ is a regular frame, then $\overline{B(L)}$ is a regular MT-algebra.\label[lemma]{regular frm to algebra}
\item If $M$ is a Hausdorff MT-algebra, then $\at(M)$ is a Hausdorff space.\label[lemma]{Hausdorff algebra to space}
\item $X$ is a Hausdorff space iff $(\P(X),\int)$ is a Hausdorff algebra.\label[lemma]{Hausdorff space to algebra}
\end{enumerate}
\end{lemma}

Let $\LCHMT$ be the full subcategory of $\LCSobMT$ consisting of locally compact Hausdorff MT-algebras.

\begin{lemma}\label{LCHMT implies spatial}
$M \in \LCHMT \implies M \in \SMT$.
\end{lemma}

\begin{proof}
Let $M \in \LCHMT$. Then $M$ is a sober MT-algebra by \cref{T2 implies T1}.
Since $M$ is locally compact, $\O(M)$ is a continuous frame by \cref{LC implies CF}, and hence $\O(M)$ is a spatial frame by \cref{continuous}.
Thus, $M$ is a spatial MT-algebra by \cref{T1 implies T1/2, when is M spatial}. 
\end{proof}

The next result generalizes the corresponding well-known topological result 
to MT-algebras.

\begin{lemma}\label{K(M) inside C(M)}
    If $M$ is a Hausdorff MT-algebra, then $\K(M)\subseteq\C(M)$.
    \end{lemma}

\begin{proof}
       Let $a \notin \C(M)$. Then 
              $\diamond a \nleq a$, so $\diamond a \wedge \neg a \neq 0$. Since $M$ is Hausdorff, there is $b \in \GC(M)$ such that $0 \neq b \leq \diamond a \wedge \neg a$. Because $b \in \GC(M)$, we have
              $b=\bigwedge\{\diamond \square c \mid b \leq \square c\}$. 
              Since $b \wedge a =0$, we have $\bigwedge \{\diamond \square c \mid b \leq \square c\} \wedge a =0$. Let $S=\{\diamond \square c \mid b \leq \square c\}$ and $T\subseteq S$ be finite. Since $S$ is directed, there is $\diamond \square c \in S$ such that $\diamond \square c \leq \bigwedge T$.        We have $0 \neq b \leq \diamond a \wedge \square c$, which implies that $\diamond a \nleq \neg \square c$. Since $\neg \square c$ is closed, this yields $a \nleq \neg \square c$, 
              and hence $ a \wedge \square c \neq 0$. Therefore, 
              $a \wedge \diamond \square c \neq 0$, 
              and so $a \wedge \bigwedge T \neq 0$. This together with $a\wedge \bigwedge S=0$ implies that $a \notin \K(M)$ (see \cref{compactelement property}).
                                              \end{proof}

\begin{lemma} \label{LCMT2 implies regular}
If $M$ is a locally compact Hausdorff MT-algebra, then $M$ is regular.
\end{lemma}

\begin{proof}
Since $M$ is Hausdorff, $M$ is a $T_1$-algebra by \cref{T2 implies T1}. 
Let $a \in \O(M)$ and let $b \in \O(M)$ with $b \nest a$. Then there is $k \in \K(M)$ such that $b \leq k \leq a$. Because $M$ is Hausdorff, \cref{K(M) inside C(M)} implies that $\diamond k=k$, so $\diamond b \leq k \leq a$. Since $M$ is locally compact, 
$a = \bigvee \{b \in \O(M) \mid b \nest a\}$, so $a = \bigvee \{b \in \O(M) \mid \diamond b \leq a\}$, and hence $M$ is regular.
\end{proof}

Let $\CRegFrm$ be the full subcategory of $\CFrm$ consisting of continuous regular frames.

\begin{lemma}\label{equivalence of LCHMT and CRegFrm}
$M \in \LCHMT \iff \O(M) \in \CRegFrm$.
\end{lemma}

\begin{proof}
First suppose $M \in \LCHMT$. Then $M$ is locally compact, so $\O(M)$ is a continuous frame by \cref{LC implies CF}. Also, $M$ is regular by \cref{LCMT2 implies regular}, and hence $\O(M)$ is a regular frame by \cref{regular algebra to frm}.
Conversely, suppose $\O(M) \in \CRegFrm$. Then $\overline{B(\O(M))}$ 
is a regular MT-algebra by \cref{regular frm to algebra}, and so a Hausdorff MT-algebra by \cref{T3 implies T2}.
Therefore, $\overline{B(\O(M))} \cong M$ by \cref{T_1/2}. Since $M$ is also sober by \cref{T2 implies T1}, $M$ is locally compact by \cref{CF implies LC}. Consequently, $M \in \LCHMT$.
\end{proof}

\cref{LCHMT implies spatial,equivalence of LCHMT and CRegFrm} 
together with \cref{SMT_LS reflective subcategory} yield:

\begin{theorem}\label{CReg LCMT2}
     The restriction $\O:\LCHMT \to \CRegFrm$ is an equivalence.
\end{theorem}

Let $\LCHaus$ be the full subcategory of $\LCSob$ consisting of locally compact Hausdorff spaces.

\begin{lemma}\label{LCHaus and LCHMT}
Let $M \in \SMT$. Then $M \in \LCHMT \iff \at(M)\in \LCHaus$.
\end{lemma}

\begin{proof}
Let $M \in \LCHMT$. Then $\at(M)$ is locally compact by \cref{LS to LC} and it is Hausdorff by \cref{Hausdorff algebra to space}. 
 Conversely, since $M$ is spatial, $M$ is isomorphic to $\P(\at(M))$ by \cref{MT equivalence}. Thus, $M$ is locally compact by \cref{LC to LS} and it is Hausdorff by \cref{Hausdorff space to algebra}.
\end{proof}

\begin{theorem}\label{dual equivalence of locally compact hausdorff spaces and MT}
    The dual equivalence between $\LCSobSMT$ and $\LCSob$ restricts to a dual equivalence between $\LCHMT$ and $\LCHaus$.
\end{theorem}

\begin{proof}
        Apply \cref{LC equivalent to LS,LCHaus and LCHMT}.
\end{proof}

\begin{corollary}
$\CRegFrm$ is dually equivalent to $\LCHaus$. 
\end{corollary}

\begin{proof}
Apply \cref{CReg LCMT2,dual equivalence of locally compact hausdorff spaces and MT}.
\end{proof}

We thus arrive at the following diagram that commutes up to natural isomorphism: 

\begin{center}
      \begin{tikzcd}[row sep=2em]
      \LCHMT \arrow[->]{rr}{\O}
      \ar[shift right=.5ex,"\at"',  shorten <= 7pt]{dr}
      & & \CRegFrm 
      \ar[dl, "pt", shift left=.5ex, shorten <= 7pt]
      \\
      &\LCHaus 
      \ar[shift right=.5ex,"\P"', shorten >= 7pt]{ul} 
      \arrow[ur, "\Omega", shift left=.5ex,  shorten >= 7pt] &
      \end{tikzcd}
\end{center}

We now restrict our attention to compact Hausdorff MT-algebras. The next lemma is a direct generalization of the corresponding result for compact Hausdorff spaces.

\begin{lemma}\label{C(M)closed}
If $M$ is a compact MT-algebra, then $\C(M)\subseteq\K(M)$.
\end{lemma}

\begin{proof}
Let $c \in \C(M)$ and $c \leq \bigvee S$ where $S \subseteq \O(M)$. Then 
$\neg c \vee \bigvee S = 1$. Since $M$ is compact, there is a finite $T \subseteq S$ such that $\neg c \vee \bigvee T = 1$. Therefore, $c \leq \bigvee T$, and hence $c\in \K(M)$.
\end{proof}

Putting \cref{K(M) inside C(M),C(M)closed} together yields the following, which generalizes a well-known result in topology to MT-algebras.

\begin{theorem}\label{C(M)=K(M)}
If $M$ is compact Hausdorff, then $\C(M)=\K(M)$.
\end{theorem}

It is well known (see, e.g., \cite[p.~150]{Engelking}) that every compact Hausdorff space is locally compact. We generalize this result to MT-algebras.

\begin{lemma}\label{compact implies LC}
If $M$ is a compact Hausdorff MT-algebra, then $M$ is locally compact.
\end{lemma}

\begin{proof}
Let $a \in \O(M)$. Since $M$ is Hausdorff, $a = \bigvee \{ b \in \GC(M) \mid b \leq a \}$. For each $b \in \GC(M)$ we have 
$b = \bigwedge \{\diamond \square c \mid b \leq \square c \}$. Therefore, if $b\le a$, then $\bigwedge \{\diamond \square c \mid b \leq \square c \} \wedge \neg a =0$. Letting $S=\{\diamond \square c \mid b \leq \square c\}$, we have $\left(\bigwedge S\right)\wedge \neg a = 0$. Since $M$ is compact and $\lnot a \in \C(M)$, by \cref{compact property} there is a finite subset $T\subseteq S$ such that $\left(\bigwedge T\right) \wedge \lnot a = 0$, so $\bigwedge T \le a$. Because $S$ is directed, $\bigwedge T \in S$, and so $\bigwedge T = \diamond \square c$ for some $c$ with $b \le \square c$. Therefore, $b \le \square c \le \diamond\square c \le a$. 
By \cref{C(M)=K(M)}, $\diamond\square c \in \K(M)$, so $\square c \nest a$. Since $b \le \square c$, we conclude that $a = \bigvee \{u \in \O(M) \mid u \nest a\}$. Thus,
$M$ is locally compact.
\end{proof}

Let $\KHMT$ be the full subcategory of $\MT$ consisting of compact Hausdorff MT-algebras, $\KHaus$ the full subcategory of $\Top$ consisting of compact Hausdorff spaces, and $\KRegFrm$ the full subcategory of $\Frm$ consisting of compact regular frames. 

Since every MT-morphism sends closed elements to closed elements (see \cref{restriction to frame homomorphism}), it follows from \cref{C(M)=K(M)} that every MT-morphism between compact Hausdorff MT-algebras is proper. 
Thus, $\KHMT$ is a full subcategory of $\LCHMT$. For the same reason, $\KHaus$ is a full subcategory of $\LCHaus$. We also have that $\KRegFrm$ is a full subcategory of $\CRegFrm$ (see, e.g., \cite[p.~8]{banaschewskistably}). 




\begin{theorem}\label{KReg KMT2}
     The restriction $\O:\KHMT \to \KRegFrm$ is an equivalence.
\end{theorem}

\begin{proof}
This follows from \cref{CReg LCMT2} since an MT-algebra $M$ is compact iff $\O(M)$ is compact. 
\end{proof}

\begin{lemma}\label{KHaus and KHMT}
Let $M \in \LCHMT$. Then $M \in \KHMT \iff \at(M)\in \KHaus$.
\end{lemma}

\begin{proof}
Let $M \in \KHMT$. Then $M$ is a Hausdorff algebra, so $\at(M)$ is a Hausdorff space by \cref{Hausdorff algebra to space}. Also, $\at(M)$ is compact by \cref{compactness of atoms}.
  Conversely, let $\at(M)\in \KHaus$. Since $M \in \LCHMT$, we have that $M$ is spatial by \cref{LCHMT implies spatial}. Therefore, $\at(M)$ being a Hausdorff space implies that $M$ is a Hausdorff algebra by \cref{Hausdorff space to algebra}. Moreover, using \cref{compactness of atoms} again yields that $M$ is compact.
\end{proof}

As an immediate consequence of \cref{KHaus and KHMT}, we obtain:

\begin{theorem}\label{dual equivalence of compact hausdorff spaces and MT}
    The dual equivalence between $\LCHMT$ and $\LCHaus$ restricts to a dual equivalence between $\KHMT$ and $\KHaus$.
\end{theorem}

\begin{corollary}[Isbell duality]
$\KRegFrm$ is dually equivalent to $\KHaus$.
\end{corollary}

\begin{proof}
Apply \cref{KReg KMT2,dual equivalence of compact hausdorff spaces and MT}.
\end{proof}

We thus arrive at the following diagram that commutes up to natural isomorphism: 

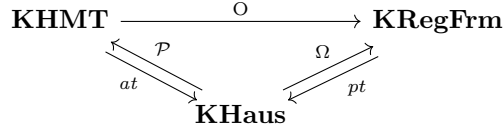
\begin{center}
      \begin{tikzcd}[row sep=2em]
      \KHMT \arrow[->]{rr}{\O}
      \ar[shift right=.5ex,"\at"',  shorten <= 7pt]{dr}
      & & \KRegFrm 
      \ar[dl, "pt", shift left=.5ex, shorten <= 7pt]
      \\
      &\KHaus 
      \ar[shift right=.5ex,"\P"', shorten >= 7pt]{ul} 
      \arrow[ur, "\Omega", shift left=.5ex,  shorten >= 7pt] &
      \end{tikzcd}
\end{center}

\section{Zero-dimensional MT-algebras}
\label{sec: 7}

In this final section we further restrict our series of equivalences and dual equivalences to the zero-dimensional setting. This, in particular, yields an alternative view of Stone duality.

For a topological space $X$, let $\CLP(X)$ be the collection of clopen (closed and open)  subsets of $X$. We recall 
(see, e.g., \cite[p.~360]{Engelking}) 
that 
$X$ is {\em zero-dimensional} if it is $T_1$ and $\CLP(X)$ is a basis for $X$.
Clearly every zero-dimensional space is Hausdorff.

The following definition is a direct generalization of the above to MT-algebras. 

\begin{definition}\label{ZeroDim}
Let $M \in \MT$. 
\begin{enumerate}[ref=\thedefinition(\arabic*)]
\item An element $a \in M$ is {\em clopen} if $a$ is both closed and open. Let $\CL(M)$ be the collection of clopen elements of $M$.
\item We call $M$ {\em zero-dimensional} if $M$ is a $T_1$-algebra and for every $a \in \O(M)$ we have
\begin{equation}
a = \bigvee \{b \in \CL(M) \mid b \leq a\} \label{zero} \tag{ZDim} 
\end{equation} \label [definition]{zerdimensional}
\end{enumerate}
\end{definition}

\begin{remark}\label{CZD}
Equivalently, a $T_1$-algebra $M$ is zero-dimensional provided for every $c \in \C(M)$ we have 
\[
c= \bigwedge \{d \in \CL(M) \mid c \leq d\}.
\]
\end{remark}

\begin{lemma}\label{restrictiontoT1/2}
An MT-algebra $M$ is zero-dimensional iff it is a $T_{1/2}$-algebra and every $a \in \O(M)$ satisfies {\em (\ref{zero})}.
\end{lemma}

\begin{proof}
Since every $T_1$-algebra is a $T_{1/2}$-algebra (see \cref{T1 implies T1/2}), it is sufficient to prove the right-to-left implication. Suppose that $M$ is a $T_{1/2}$-algebra and every $a \in \O(M)$ satisfies (\ref{zero}). It is enough to show that $M$ is $T_1$. Let $0\neq a \in M$. Since $M$ is a $T_{1/2}$-algebra, there exist $u \in \O(M)$ and $f \in \C(M)$ such that $0\neq u \wedge f 
\leq a$. The condition $u \wedge f \neq 0$ implies $u \nleq \neg f$. By (\ref{zero}), there is $c \in \CL(M)$ such that $c \leq u$ and $c \nleq \neg f$. Therefore, $0 \neq c \wedge f$,  $c \wedge f \in \C(M)$, and $c \wedge f \leq a$. Thus, $M$ is a $T_1$-algebra. 
 \end{proof}

\begin{remark}
Whether $M$ being a $T_{1/2}$-algebra can further be weakened to being a $T_0$-algebra remains open. Note that this is the case for spatial MT-algebras since every $T_0$-space satisfying (\ref{zero}) is a $T_1$-space.
\end{remark}

We recall 
that an element $c$ of a frame $L$ is {\em complemented} if $c\vee c^*=1$, where 
\[
c^* := \bigvee \{x \in L \mid c \wedge x=0\}
\] 
is the pseudocomplement of $c$. Let $\cmp(L)$ be the collection of complemented elements of $L$. 
Then $L$ is {\em zero-dimensional} 
if $a=\bigvee \{c \in \cmp(L) \mid c \leq a \}$ for each $a\in L$ (see, e.g., \cite[p.~258]{banaschewski1989universal}). 
 
\begin{lemma}\label{clopens are complemented}
Let $M \in \MT$. Then $a \in \CL(M)$ iff $a \in \cmp(\O(M))$.
\end{lemma}

\begin{proof}
First suppose $a\in \CL(M)$. Then $a, \neg a \in \O(M)$ and 
$a \vee \neg a =1$. Therefore, $a \in \cmp(\O(M))$. Conversely, suppose $a \in \cmp(\O(M))$. Then $a \wedge a^*=0$ and $a\vee a^*=1$.
Thus, $\neg a= a^*$, 
so $\neg a \in \O(M)$, and hence $a \in \CL(M)$. 
\end{proof}

As an immediate consequence of \cref{clopens are complemented}, we obtain:

\begin{proposition}\label{zero-dimensional}
Let $M$ be a $T_1$-algebra.
Then $M$ is a zero-dimensional MT-algebra iff $\O(M)$ is a zero-dimensional frame.
\end{proposition}

Let $\ZDMT$ be the full subcategory of $\MT$ consisting of zero-dimensional MT-algebras, and let $\ZDFrm$ be the full subcategory of $\Frm$ consisting of zero-dimensional frames.

\begin{lemma}
The restriction $\O:\ZDMT \to \ZDFrm$ is well defined and essentially surjective.
\end{lemma}

\begin{proof}
That the restriction $\O:\ZDMT \to \ZDFrm$ is well defined follows from \cref{zero-dimensional}. To see that it is essentially surjective, let $L \in \ZDFrm$. Then $\O\left(\overline{B(L)}\right)=L$ by \cref{thm: essentially surj}. Therefore, $\overline{B(L)}$ satisfies (\ref{zero}) by \cref{clopens are complemented} and $\overline{B(L)}$ is a $T_{1/2}$-algebra by \cref{T_1/2}.
Therefore, by \cref {restrictiontoT1/2}, we conclude that $\overline{B(L)}$ is a $T_1$-algebra, and hence $\overline{B(L)} \in \ZDMT$. Thus, $\O:\ZDMT \to \ZDFrm$ is  essentially surjective.
\end{proof}

We recall (see, e.g., \cite[p.~8]{bezhanishvili2023structure}) that a space $X$ is \emph{locally Stone} if it is zero-dimensional, locally compact, and Hausdorff, and that a frame $L$ is \emph{locally Stone} if it is continuous and zero-dimensional. 
Let $\LStone$ be the full subcategory of $\LCHaus$ consisting of zero-dimensional spaces. Let also $\LStoneFrm$ be the full subcategory of $\CRegFrm$ consisting of locally Stone frames.

\begin{definition}
An MT-algebra $M$ is \emph{locally Stone} if it is zero-dimensional, locally compact, and Hausdorff. Let $\LStoneMT$ be the full subcategory of $\LCHMT$ consisting of locally Stone algebras.
\end{definition}

\begin{theorem}\label{LStoneMT to LStoneFrm}
The restriction $\O:\LStoneMT \to \LStoneFrm$ is an equivalence.
\end{theorem}

\begin{proof}
Let $M \in \LCHMT$. By \cref{equivalence of LCHMT and CRegFrm,zero-dimensional}, $M \in \LStoneMT$ iff $\O(M) \in \LStoneFrm$. It remains to apply \cref{CReg LCMT2}.
\end{proof}

\begin{lemma}\label{clopens in M and at(M)}
Let $M \in \SMT$. Then $a \in \CL(M) \iff \eta(a) \in \CLP(\at(M))$.
\end{lemma}

\begin{proof}
The left-to-right implication follows from the definition of the topology on $\at(M)$. 
For the converse implication, let $\eta(a)\in \CLP(\at(M))$. Then there exist $u \in \O(M)$ and $f \in \C(M)$ such that $\eta(a)=\eta(u) =\eta(f)$. But $\eta$ is an isomorphism because $M \in \SMT$, so $a = u = f$, and hence $a \in \CL(M)$. 
\end{proof}

\begin{lemma} \label{lem: LStone}
Let $M \in \LCHMT$. Then $M \in \LStoneMT$ iff $\at(M)\in \LStone$.
\end{lemma}

\begin{proof}
Since $M \in \LCHMT$, $M$ is spatial by \cref{LCHMT implies spatial}, and so $\eta$ is an isomorphism by \cref{MT adjunction}.
Therefore, for each $a \in \O(M)$, we have 
\begin{align*}
 &\eta(a)= \bigcup \{\eta(b)\in \CLP(\at(M)) \mid \eta(b)\subseteq\eta(a) \} \\
\iff &\eta(a)=\eta\left(\bigvee \{b \in \CL(M) \mid \eta(b)\subseteq\eta(a)\}\right) && \text{by \cref{clopens in M and at(M)}}\\
\iff &a= \bigvee \{b \in \CL(M) \mid b\le a\}&& \text{$\eta$ is an isomoprhism}.
\end{align*} 
Thus, $M$ is a zero-dimensional MT-algebra iff $\at(M)$ is a zero-dimensional space. This together with \cref{LCHaus and LCHMT} yields that $M \in \LStoneMT$ iff $\at(M)\in \LStone$.
\end{proof}

As an immediate consequence of \cref{lem: LStone}, we obtain:

\begin{theorem}\label{LStoneMT to LStone}
The dual equivalence between $\LCHMT$ and $\LCHaus$ restricts to a dual equivalence between $\LStoneMT$ and $\LStone$.
\end{theorem}

Putting \cref{LStoneMT to LStoneFrm,LStoneMT to LStone} together, we arrive at the following diagram that commutes up to natural isomorphism: 

\begin{center}
      \begin{tikzcd}[row sep=2em]
      \LStoneMT \arrow[->]{rr}{\O}
      \ar[shift right=.5ex,"\at"',  shorten <= 7pt]{dr}
      & & \LStoneFrm 
      \ar[dl, "pt", shift left=.5ex, shorten <= 7pt]
      \\
      &\LStone 
      \ar[shift right=.5ex,"\P"', shorten >= 7pt]{ul} 
      \arrow[ur, "\Omega", shift left=.5ex,  shorten >= 7pt] &
      \end{tikzcd}
\end{center}

Let $\StoneMT$ be the full subcategory of $\LStoneMT$ consisting of compact MT-algebras and let $\StoneFrm$ be the full subcategory of $\LStoneFrm$ consisting of compact frames. Let also $\Stone$ be the full subcategory of $\LStone$ consisiting of compact spaces.
Observe that $\StoneMT$ is the full subcategory of $\KHMT$ consisting of zero-dimensional MT-algebras, that $\StoneFrm$ is the full subcategory of $\KRegFrm$ consisting of zero-dimensional frames, and that $\Stone$ is the full subcategory of $\KHaus$ consisting of zero-dimensional spaces.
Therefore, by combining \cref{dual equivalence of locally compact hausdorff spaces and MT,KReg KMT2,dual equivalence of compact hausdorff spaces and MT,LStoneMT to LStoneFrm,LStoneMT to LStone},  we obtain:

\begin{theorem} \label{StoneMT and StoneFrm}
$\StoneMT$ is equivalent to $\StoneFrm$ and dually equivalent to $\Stone$.
\end{theorem}

Thus, the following diagram commutes up to natural isomorphism: 
\begin{center}
      \begin{tikzcd}[row sep=2em]
      \StoneMT \arrow[->]{rr}{\O}
      \ar[shift right=.5ex,"\at"',  shorten <= 7pt]{dr}
      & & \StoneFrm 
      \ar[dl, "pt", shift left=.5ex, shorten <= 7pt]
      \\
      &\Stone 
      \ar[shift right=.5ex,"\P"', shorten >= 7pt]{ul} 
      \arrow[ur, "\Omega", shift left=.5ex,  shorten >= 7pt] &
      \end{tikzcd}
\end{center}

We conclude the paper by showing how Stone duality for boolean algebras fits in the above diagram. For this we recall the notion of the canonical extension of a boolean algebra \cite[p.~908]{jonsson1951boolean}: 

\begin{definition}
The {\em canonical extension} of a boolean algebra $B$ is a pair $(B^\sigma,e)$ such that $B^\sigma$ is a complete boolean algebra and $e:B \to B^\sigma$ is a boolean embedding such that 
the following holds 
\begin{enumerate}[ref=\thedefinition(\arabic*)]
\item each element of $B^\sigma$ is a join of meets of elements of $e[B]$.\label[definition]{dense} 
\item for each $S, T \subseteq B$, from $\bigwedge e[S] \leq \bigvee e[T]$ it follows that $\bigwedge S_0 \le \bigvee T_0$ for some finite $S_0 \subseteq S$ and $T_0 \subseteq T$.\label[definition]{compact}
\end{enumerate}
\end{definition}

We identify $B$ with its image $e[B]$ and view $B$ as a subalgebra of $B^\sigma$. 
It is well known \cite[p.~910]{jonsson1951boolean}
that the canonical extension of $B$ is unique, up to isomorphism, and that $B^\sigma$ can be constructed as the powerset of the set of ultrafilters of $B$. Other (choice-free and point-free) constructions of $B^\sigma$ can be found in \cite{GehrkeHarding2001,Holliday,bezhanishvili2023point}.

Following \cite[Def.~2.4]{BMM},
for a boolean algebra $B$, define $\square: B^{\sigma} \to B^{\sigma}$ by 
\[
\square a=\bigvee \{b \in B \mid b \leq a\}. \tag{\dag} \label{topo completion}
\] 
By \cite[Prop.~2.5]{BMM}, $(B^{\sigma},\square)\in \MT$
and $\diamond a= \bigwedge \{c \in B \mid a \leq c\}$.

\begin{lemma} \label{lem: BA to MT}
For a boolean algebra $B$ we have that $M:=(B^{\sigma},\square)$ is a Stone MT-algebra and $\CL(M)$ is isomorphic to $B$.
\end{lemma}

\begin{proof}
We identify $B$ with its image in $B^{\sigma}$ and show that $\CL(M)=B$. Since for $a \in B$, we have $\square a = a = \diamond a$, we see that $B \subseteq \CL(M)$. For the reverse inclusion, let $a \in \CL(M)$. Then 
\[
a=\square a = \bigvee A \quad \mbox{and} \quad a =\diamond a = \bigwedge C,
\] 
where $A=\{b \in B \mid b \leq a\}$ and $C= \{c \in B \mid a \leq c\}$. By \cref{compact}, 
there are finite $A_0 \subseteq A$ and $C_0 \subseteq C$ such that 
$a=\bigwedge C \leq \bigwedge C_0 \le \bigvee A_0\leq \bigvee A=a$. Thus, $a \in B$. 

We next show that $M$ is a Stone MT-algebra. 
To see that $M$ is compact, let $1=\bigvee S$ where $S \subseteq \O(M)$. Since each element of $S$ is a join from $B$, we obtain that $1=\bigvee T$ for some $T\subseteq B$. By \cref{compact}, $1=\bigvee T_0$ for some finite $T_0 \subseteq T$.
Therefore, $M$ is compact. We show that $M$ is Hausdorff. By \cref{dense},
each $a \in M$ is a join of meets of elements of $B$.
Since $\diamond \square b=b$ for each $b \in B$, we see that $a$ is a join of elements of $\GC(M)$. 
Thus, $M$ is Hausdorff. Finally, $M$ is zero-dimensional because $\CL(M)=B$ and each element of $\O(M)$ is a join of elements of $B$.
\end{proof}

Let $\BA$ be the category of boolean algebras and boolean homomorphisms. Associating with each MT-algebra $M$ the boolean algebra $\CL(M)$ of clopen elements of $M$ and with each MT-morphism $f: M\to N$ its restriction $f| _{\CL(M)} : \CL(M)\to\CL(N)$ defines a covariant functor $\CL:\StoneMT \to \BA$.

\begin{theorem}\label{StoneMT and BA}
$\CL:\StoneMT \to \BA$ is an equivalence.
\end{theorem}

\begin{proof}
Let $B \in \BA$. By \cref{lem: BA to MT}, 
$M:=(B^{\sigma},\square)\in \StoneMT$ and $\CL(M) \cong B$. 
Therefore, $\CL$ is essentially surjective. To see that $\CL$ is faithful, 
let $f,g :M\to N$ be MT-morphisms between $M,N\in\StoneMT$ such that
$f|_{\CL(M)}= g|_{\CL(M)}$. We show that $f=g$.
    Let $a \in M$. 
        Since $M$ is a $T_1$-algebra, $a=\bigvee C$ for some $C\subseteq\C(M)$. Thus, 
    \begin{align*}
    f(a) &= f\left(\bigvee C\right) = \bigvee f[C] \\ 
            &= \bigvee \left\{f\left(\bigwedge \{b \in \CL(M) \mid b \ge c \right) \biggm| c \in C \right\}&\text{by \cref{CZD}} \\
    &= \bigvee \bigwedge \left\{ f(b) \mid b \in \CL(M), b \ge c, c \in C \right\}  \\
    &= \bigvee \bigwedge \{g(b) \mid  b \in \CL(M), b \ge c, c \in C \} \\
    &= \bigvee \left\{ g\left(\bigwedge \{b \in \CL(M) \mid  b \ge c \}\right) \biggm| c \in C \right\} &\text{by \cref{CZD}} \\
    &= \bigvee g[C] = g\left(\bigvee C\right) \\         &= g(a),
    \end{align*} 
        yielding that $\CL$  is faithful. 
        It is left to show that $\CL$ is full. Let $h:A \to B$ be a boolean homomorphism. We identify $A$ and $B$ with their images in $A^{\sigma}$ and $B^{\sigma}$, respectively. By \cite[Thm.~2.16]{jonsson1951boolean}, there is a complete boolean homomorphism $h^{\sigma}:A^{\sigma} \to B^{\sigma}$ such that $h$ is the restriction of $h^\sigma$ to $A$. Let $M=(A^{\sigma},\square_A)$ and $N= (B^{\sigma},\square_B)$ be the corresponding Stone MT-algebras (see \cref{lem: BA to MT}). By \cite[Thm.~2.7(2)]{BMM}, $h^{\sigma}:M \to N$ is an MT-morphism. 
            Consequently, $\CL$ is an equivalence.
\end{proof}

\begin{remark}
A quasi-inverse of 
$\CL:\StoneMT\to\BA$ 
is the functor $(-)^\sigma:\BA\to\StoneMT$, which associates with each boolean algebra $B$ the Stone MT-algebra $M=(B^\sigma,\square)$ 
and with each boolean homomorphism $h : B_1 \to B_2$ the MT-morphism $h^\sigma : M_1 \to M_2$.
\end{remark}

Putting \cref{StoneMT and StoneFrm,StoneMT and BA} together yields: 

\begin{corollary}[Stone duality]
$\BA$ is dually equivalent to $\Stone$.
\end{corollary}

We thus arrive at the following diagram that commutes up to natural isomorphism: 

\begin{center}
      \begin{tikzcd}[row sep=4em]
      \StoneMT \arrow[->]{rr}{\O}
      \arrow[drr, shift right=.5ex, "\at"', shorten <= 7pt]
      \ar[d, shift right=.5ex, "\CL"']
      \ar[d, <-, shift left=.5ex, "(-)^\sigma"]& & \StoneFrm
      \ar[d, "\pt", shift left=.5ex]\\
       \BA \ar[rr, shift right=.5ex, ->, "\Uf"']
       \ar[rr, shift left=.5ex, <-, "\CLP"]  &&\Stone 
       \arrow[shift right=.5ex,"\P"' , shorten >= 7pt]{ull}
       \arrow[u, "\Omega", shift left=.5ex]
      \end{tikzcd}
\end{center}

In the above diagram, $\CLP$ and $\Uf$ are the contravariant functors of Stone duality, so $\Uf$ maps a boolean algebra $B$ to its space of ultrafilters.
We already saw that the upper triangle of the diagram commutes up to natural isomoprhism.
To see that the lower triangle commutes, observe that 
$\P \circ \Uf \cong (-)^\sigma$ (since the canonical extension is isomorphic to
the powerset of the set of ultrafilters)
and $\CLP \circ \at \cong \CL$ (by \cref{clopens in M and at(M),StoneMT and StoneFrm}).

We conclude with the diagram that summarizes all equivalences and dual equivalences discussed in this paper. The notation `$A \hookrightarrow B$' stands for `$A$ is a full subcategory of $B$,'
`$A \longleftrightarrow B$'  
for `$A$ is equivalent to $B$,' and `$A \mathrel{\tikz[baseline]{\draw[<->](0,.58ex)-- node[label={[label distance=-5pt]$d$}]{} (.7,.58ex)}}B$' 
for `$A$ is dually equivalent to $B$.'\\

\begin{center}
\begin{tikzcd}[ampersand replacement=\&,row sep={30,between origins}, column sep={50,between origins}]
 \LCSobSMT \ar[hookleftarrow]{dd} \ar[<->]{rr}    
    \&\& \CFrm \ar[hookleftarrow]{dd}\ar[<->,"d"]{rr}
    \&\& \LCSob \ar[hookleftarrow]{dd}  \\\\
 \StLCSMT \ar[hookleftarrow]{dr} \ar[hookleftarrow]{dd}\ar[<->]{rr}
    \&\& \StCFrm \ar[hookleftarrow]{dr} \ar[hookleftarrow]{dd}\ar[<->,"d"]{rr}
    \&\& \StLCSob \ar[hookleftarrow]{dr} \ar[hookleftarrow]{dd} \\
 \& \StKSMT \ar[<->, crossing over]{rr}
    \&\& \StKFrm \ar[<->, crossing over, "d"]{rr}
    \&\& \StKSob \ar[hookleftarrow]{dd} \\
 \LCHMT \ar[hookleftarrow]{dr}\ar[<->]{rr}
    \&\& \CRegFrm \ar[hookleftarrow]{dr}\ar[<->,"d"]{rr}
    \&\& \LCHaus \ar[hookleftarrow]{dr} \\
 \& \KHMT \ar[hookleftarrow]{dd}\ar[<->]{rr}
    \&\& \KRegFrm \ar[hookleftarrow]{dd}\ar[<->,"d"]{rr}
    \&\& \KHaus \ar[hookleftarrow]{dd}  \\\\
 \& \StoneMT \ar[<->]{rr}
    \&\& \StoneFrm\ar[<->,"d"]{rr}
    \&\& \Stone
\ar[hookleftarrow, crossing over, from=4-2, to=6-2]
\ar[hookleftarrow, crossing over, from=4-4, to=6-4]
\end{tikzcd}
\end{center}

In the tables below we summarize the categories of frames, topological spaces, and MT-algebras considered in this paper.

\begin{table}[H]
    \begin{tabular}{lll}
        \toprule
        \bf Category & \bf Objects & \bf Morphisms \\ \midrule
        \Frm & Frames & Frame homomorphisms\\
        \SFrm& Spatial frames & Frame homomorphisms\\
        \CFrm  & Continuous frames        & Proper frame homomorphisms \\
        \StCFrm & Stably continuous frames & Proper frame homomorphisms \\
        \StKFrm & Stably compact frames    & Proper frame homomorphisms \\
        \CRegFrm & Continuous regular frames    & Proper frame homomorphisms \\
        \KRegFrm & Compact regular frames    & Frame homomorphisms \\
        \LStoneFrm & Locally Stone frames    & Proper frame homomorphisms \\
        \StoneFrm & Stone frames    & Frame homomorphisms \\
 \bottomrule
    \end{tabular}
    \caption{Categories of frames.\label{table:frames}}
\end{table}

\begin{table}[H]
\begin{tabular}{lll}
    \toprule 
    \bf Category & \bf Objects & \bf Morphisms \\ \midrule
    \Top & Topological spaces & Continuous maps\\
    \Sob & Sober spaces & Continuous maps\\
    \LCSob & Locally compact sober spaces  & Proper maps \\
    \StLCSob & Stably locally compact spaces & Proper maps \\
    \StKSob  & Stably compact spaces         & Proper maps \\ 
    \LCHaus  & Locally compact Hausdorff spaces         & Proper maps \\ 
    \KHaus  & Compact Hausdorff spaces         & Continuous maps \\
    \LStone  & Locally Stone spaces         & Proper maps \\  
    \Stone  & Stone spaces         & Continuous maps \\ 
    \bottomrule
\end{tabular}
\caption{Categories of spaces.\label{table:spaces}}
\end{table}

\begin{table}[H]
\begin{tabular}{lll}
    \toprule 
    \bf Category & \bf Objects & \bf Morphisms \\ \midrule
    \MT & MT-algebras & MT-morphisms\\
    \SMT & Spatial MT-algebras & MT-morphisms\\
    \SobMT & Sober MT-algebras & MT-morphisms\\
    \LCSobMT & Locally compact sober MT-algebras  & Proper MT-morphisms \\
    \StLCMT & Stably locally compact MT-algebras & Proper MT-morphisms \\
    \StKMT  & Stably compact MT-algebras        & Proper MT-morphisms \\ 
    \LCHMT & Locally compact Hausdorff MT-algebras & Proper MT-morphisms\\
    \KHMT & Compact Hausdorff MT-algebras & MT-morphisms\\
    \LStoneMT & Locally Stone MT-algebras & Proper MT-morphisms\\
    \StoneMT & Stone MT-algebras & MT-morphisms\\  
    \bottomrule
\end{tabular}
\caption{Categories of MT-algebras.\label{table:algebras}}
\end{table}

We note that the dualities established in this paper that involve proper maps have obvious counterparts, where proper frame homomorphisms are replaced with frame homomorphisms, proper maps with continuous maps, and proper MT-morphisms with MT-morphisms.  

\bibliographystyle{alpha-init}
\bibliography{references}

\end{document}